\documentclass[11pt]{article}
\usepackage{latexsym}
\usepackage{amssymb}
\usepackage{graphicx}
\usepackage{enumerate}
\usepackage{amssymb}
\usepackage{amsmath}
\usepackage{color}
\usepackage{tikz}
\usepackage{hyperref}
\hypersetup{hypertex=true,
	colorlinks=true,
	linkcolor=blue,
	anchorcolor=blue,
	citecolor=blue}
\usepackage{amsthm}
\usepackage{cleveref}
\crefformat{section}{\S#2#1#3} 
\crefformat{subsection}{\S#2#1#3}
\crefformat{subsubsection}{\S#2#1#3}
\usepackage{tikz}
\theoremstyle{plain}
\newtheorem{thm}{Theorem}[section]
\newtheorem{lem}[thm]{Lemma}
\newtheorem{cor}[thm]{Corollary}
\newtheorem{prop}[thm]{Proposition}

\theoremstyle{definition}
\newtheorem{rem}[thm]{Remark}

\newtheorem{defn}[thm]{Definition}
\def\Q{\tilde{Q}}
\def\tP{\tilde{\P}}

\def\pT{\frac{\p}{\p T}}
\def\tD{{\tilde{\Delta}}}

\def\tl{{\tilde{\l}}}

\def\H{\mathcal{H}}
\def\lan{\langle}
\def\ran{\rangle}
\def\Re{\mathrm{Re}}
\def\la{\mathrm{la}}
\def\mL{\mathrm{L}}
\def\cE{{\mathcal{E}}}
\def\mS{\mathrm{S}}
\def\sm{\mathrm{sm}}

\def\End{\mathrm{End}}

\def\bE{{\bar{E}}}

\def\rank{\mathrm{rank}}

\def\tf{\tilde{f}}
\def\im{\mathrm{im}}
\def\rel{\mathrm{rel}}
\def\bd{\mathrm{bd}}

\def\lan{\langle}
\def\ran{\rangle}

\def\abs{\mathrm{abs}}
\def\Tr{\operatorname{Tr}}

\def\bz{{\bar{Z}}}
\def\bm{\bar{M}}

\def\Z{\mathbb{Z}}

\def\P{\mathcal{P}}
\def\dvol{\mathrm{dvol}}

\def\l{\lambda}
\def\a{\alpha}

\def\b{\beta}

\def\R{\mathbb{R}}
\def\F{\bar{F}}

\def\C{\mathbb{C}}

\def\de{\delta}
\def\half{\frac{1}{2}}

\def\tC{{\tilde{C}}}
\def\ttD{{\tilde{D}}}

\def\T{\mathcal{T}}

\begin{document}
	
	\def\o{\omega}
	\def\O{\Omega}
	\def\p{\partial}
	\def\s{{\theta}}
	\def\a{\alpha}
	\def\half{\frac{1}{2}}
	\def\l{\lambda}
	\def\dist{\mathrm{dist}}
	\def\Zs{{Z_\s}}
	\def\Zso{{Z_{\s,1}}}
	\def\Zsw{{Z_{\s,2}}}
	\def\bzs{{\bz_\s}}
	\def\bzso{{\bz_{\s,1}}}
	\def\bzsw{{\bz_{\s,2}}}
	
	\newcommand{\be}{\begin{equation}}
		\newcommand{\ee}{\end{equation}}
	\newcommand{\bc}{\begin{cases}}
		\newcommand{\ec}{\end{cases}}
	\newcommand{\bes}{\begin{equation*}}
		\newcommand{\ees}{\end{equation*}}
	\newcommand{\ba}{\begin{align}}
		\newcommand{\ea}{\end{align}}
	\newcommand{\bas}{\begin{align*}}
		\newcommand{\eas}{\end{align*}}
	\newcommand{\es}{\end{split}}
\newcommand{\bs}{\begin{split}}
	\title{A new proof of gluing formula for analytic torsion forms}
	
	
	\author{Junrong Yan
		\footnote{Beijing International Center for Mathematical Research, Peking University, Beijing, China 100871, j\_yan@bicmr.pku.edu.cn. Supported by Boya Postdoctoral Fellowship at Peking University.}
	}
	
	\maketitle
	\begin{abstract}{
			By extending the author's prior work  \cite{Yantorsions}  to the family case, this paper presents a new proof of the gluing formula for analytic torsion forms, considerably simplifying the proof given by Puchol-Zhang-Zhu \cite{puchol2020adiabatic}. 
			It is expected that this new proof will enhance the exploration of the higher Cheeger-M\"uller theorem
			for general flat bundles.
		}
	\end{abstract}
	
	\tableofcontents
	
	\section{Introduction}

	\subsection{Overview}
	Let $(M,g)$ be a closed Riemannian manifold associated with a flat complex vector bundle $F\to M$ with a Hermitian metric $h^F$.
	The corresponding Ray-Singer analytic torsion \cite{ray1971r} is the determinant of the Hodge-Laplacian on $F$-valued differential forms. 
	The Ray-Singer analytic has a well-known topological counterpart, the Reidemeister torsion (R-torsion)\cite{reidemeister1935homotopieringe}. According to the well-known Cheeger-M\"uller/Bismut-Zhang theorem, the two torsions are related \cite{cheeger1979analytic,muller1978analytic,muller1993analytic,bismutzhang1992cm}.  
	
	It was conjectured that R-torsion and Ray-Singer torsion can be extended to invariants of a $C^\infty$ fibration  $\pi: M \rightarrow S$ of closed fiber $Z$, associated with a flat complex vector bundle $F\to M$ \cite{wagoner1976diffeomorphisms}.  Bismut and Lott \cite{bismut1995flat} then construct analytic torsion forms (BL-torsion), which are even forms on $S$. Igusa, motivated by the work of Bismut and Lott, developed the Igusa-Klein torsion, a higher topological torsion (IK-torsion)  \cite{igusa2002higher}. As an application of IK-torsion,  Goette, Igusa, and Williams \cite{goette2014exotic1,goette2014exotic} uncover fiber bundles' exotic smooth structure. Then it becomes a natural and significant question to investigate the connection between these higher torsion invariants. Under the assumption that there exists a fiberwise Morse function \cite{bismut2001families, goette2001morse}, Bismut and Goette established a higher version of the Cheeger-M\"uller/Bismut-Zhang theorem. Lastly, an interesting relation between the Bismut-Freed connection and analytic torsion forms was observed by Dai and Zhang \cite{DaiZhang}.

	Higher torsion invariants were axiomatized by Igusa \cite{igusa2008axioms}, and Igusa showed that IK-torsion complies with his axioms. Two of Igusa's axioms are the additivity axiom and the transfer axiom. And any higher torsion invariant that satisfies Igusa's axioms is simply a linear combination of IK-torsion and the higher Miller-Morita-Mumford class \cite{morita1987characteristic,mumford1983towards, miller1986homology}. BL-torsion is proven to satisfy the transfer axiom thanks to the work of Ma \cite{ma1999functoriality}. 
	At a conference\footnote{Smooth Fibre Bundles and Higher Torsion Invariants, http://www.uni-math.gwdg.de/wm03/,
		G\"ottingen, 2003} on higher torsion invariants in 2003, the gluing formula of analytic torsion forms was posed as an open problem. A gluing formula for analytic torsion forms was recently established by Puchol-Zhang-Zhu in \cite{puchol2020adiabatic}. Using \cite{ma1999functoriality,puchol2020adiabatic} and Igusa's axiom of higher torsion invariants, Puchol-Zhang-Zhu were able to prove the higher Cheeger-M\"uller theorem for trivial bundles in \cite{puchol2021comparison}. This paper presents a different and considerably less complicated proof of the gluing formula in the hope that this will offer insight into the study of higher Cheeger-M\"uller theorem for general flat bundles.
	
	Let's assume that $N\subset M$ is a hypersurface that fiberwisely divides $Z$ into two parts, $Z_1$ and $Z_2$, and that it also divides $M$ into two pieces, $M_1$ and $M_2$. In this paper, we prove a gluing formula using the Witten deformation for non-Morse functions. We choose a family of smooth functions $p_T$ where $\lim_{T\to\infty}p_T$ contains critical loci $M_1$ and $M_2$ with ``Morse indices" $0$ and $1$ respectively. Assuming $T$ varies from $0$ to $\infty$, we may adopt the philosophy of Witten deformation to see the relationship between the analytic torsion forms on $M$ and analytic torsion forms on the two components above. 
	
	Lastly, it is worth mentioning that Puchol-Zhang-Zhu in \cite{puchol2020adiabatic} constructed a similar family of non-Morse functions to ensure that their two parametrized Laplacian operators have a uniform spectral gap near 0 when the adiabatic limit is applied.
	\ \\
	
	{\em Acknowledgment:  
		The author is appreciative of Professor Xianzhe Dai's consistently stimulating conversation and encouragement. The author also appreciates the insightful discussion with Martin Puchol and Yeping Zhang.}
	
	\subsection{Main results}
	Let $M\to S$ be a fibration of closed fiber $Z$.  Let $T Z\subset TM$ be the vertical tangent bundle of the fiber bundle with metric $g^{TZ}$. Let $F\to M$ be a flat complex vector bundle with Hermitian metric $h^F$, and $\nabla^F$ be a flat connection on $F$. Fix a splitting of $TM$
	\[T M=T^H M \oplus T Z.\]
	
	Let $Q^{\mS}$ be the space of closed even forms, $Q^{\mS}_0$ be the space of exact even forms.
	Based on the data $(T^HM,g^{TZ},h^F)$, one can define analytic torsion forms $$\T(T^HM,g^{TZ},h^F)\in Q^{\mS}.$$
	See Definition \ref{torc} for the definition of $\T(T^HM,g^{TZ},h^F).$

	Let $N \subseteq M$ be a hypersurface transversal to $Z$. We suppose that $\pi|_N: N \rightarrow S$ is a fibration of fiber $Y:=N \cap Z$. Suppose that $N$ cuts $M$ into two pieces $M_1$ and $M_2$, and fiberwisely, cut $Z$ into two pieces $Z_1$ and $Z_2$. Let $\pi_i: M_i \rightarrow S$ be the restriction of $\pi$ to $M_i$. Then $\pi_i$ is a fibration of $Z_i$ ($i=1,2$). Let $F_i, T^HM_i$ and $h^{F_i}$ be the restriction of $F, T^HM$ and $h^{F}$ to $M_i$ respectively ($i=1,2$).
	
	We put the relative boundary condition on the boundary of $Z_1$ and the absolute boundary condition on the boundary of $Z_2$ (see \cref{defnblbd}). The analytic torsion form for fibration $M_i$ with boundary equipped with absolute/relative boundary condition could be defined, denoted by $\T_i(T^HM_i,g^{TZ_i},h^{F_i})\in Q^{\mS}$ ($i=1,2$). See \cref{defnblbd} for more details.
	
	For $\s\in S$, let $Z_\s:=\pi^{-1}\s, F_\s:=F|_{Z_\s}.$ Let $H\left(Z, F\right)$ be the $\mathbb{Z}$-graded vector bundle over $S$ whose fiber over $\s \in S$ is the cohomology $H\left(Z_\s, F_{\s}\right)$ of the sheaf of locally flat sections of $F$ on $Z_\s$.
	
	We have a Mayer-Vietoris exact sequence of flat complex vector bundles over $S$, 
	
	\be\label{exactse}\cdots \rightarrow H^k_{\rel}\left(Z_2;F_2\right) \rightarrow H^k\left(Z; F\right) \rightarrow H^k_{\abs}\left(Z_1;F_1\right) \rightarrow \cdots\ee
	with metric and flat connections induced by Hodge theory (c.f. \cite{bismut1995flat}).
	
	Let $\T\in Q^{\mS} $ be the torsion form for the exact sequence (\ref{exactse})(c.f. \cref{defnblfn}). We assume that $g^{TZ},h^{F}$ and $T^HM$ are product-type near $N$, then
	\begin{thm}\label{main0}
		In $Q^{\mS} / Q^{S}_0$,
		\begin{align*}
			&\ \ \ \ \mathcal{T}(T^HM,g^{TZ},h^F)-\T_1(T^{H}M_1,g^{TZ_1},h^{F_1})-\T_2(T^HM_2,g^{TZ_2},h^{F_2})-\mathcal{T}\\
			&=\frac{1}{2} \chi(Y) \rank(F) \log 2.\end{align*}
	\end{thm}
	For any closed oriented submanifold $S' \subseteq S$, the map
	$$
	\int_{S'}: Q^{\mS} \rightarrow \mathbb{R}
	$$
	descends to a linear function on $Q^{\mS} / Q^{S}_0$. It follows from Stokes' formula and the de Rham theorem that these linear functions separate the elements of $Q^{\mS} / Q^{S, 0}$. To prove Theorem \ref{main0}, it is, therefore, sufficient to show the case in which $S$ is a closed manifold.

	\subsection{Main ideas}\label{idea}
	
	Let $N\subset M$ be a hypersurface cutting $M$ into two pieces $M_1$ and $M_2$ (see Figure \ref{fig1}). Let $U$ be a neighbor of $N$, such that $U\cap M_1$ is diffeomorphic to $(-2,-1]\times N$ and identify $\p M_1$ with $\{-1\}\times N$, $U\cap M_2$ is diffeomorphic to $[1,2)\times N$ and identify $\p M_2$ with $\{1\}\times N$. Moreover, assume that on $U$, $ T^HM, g^{TZ}$ and $h^F$ are product-type.
	
	\begin{figure}[h]
		\setlength{\unitlength}{0.75cm}
		\centering
		\begin{picture}(18,6.5)
			\qbezier(3,2)(5,1)(7,2) 
			\qbezier(3,4)(5,5)(7,4) 
			\qbezier(3,2)(1.5,3)(3,4) 
			\qbezier(4,3.2)(5,2.7)(6,3.2) 
			\qbezier(4.5,3)(5,3.2)(5.5,3) 
			\qbezier(7,2)(7.5,2.3)(8,2.4) 
			\qbezier(7,4)(7.5,3.7)(8,3.6) 
			\qbezier(11,2)(13,1)(15,2) 
			\qbezier(11,4)(13,5)(15,4) 
			\qbezier(15,2)(16.5,3)(15,4) 
			\qbezier(12,3.2)(13,2.7)(14,3.2) 
			\qbezier(12.5,3)(13,3.2)(13.5,3) 
			\qbezier(11,2)(10.5,2.3)(10,2.4) 
			\qbezier(11,4)(10.5,3.7)(10,3.6) 
			\qbezier(8,2.4)(9,2.4)(10,2.4) 
			\qbezier(8,3.6)(9,3.6)(10,3.6) 
			
			\qbezier(9,2.4)(8.7,3)(9,3.6) 
			\qbezier[30](9,2.4)(9.3,3)(9,3.6) 
			
			\put(9,5.3){\vector(0,-1){1.5}} 
			
			\put(3,4.8){$\overbrace{\hspace{40mm}}^{}$} \put(5.1,5.4){$M_{1}$}  
			\put(9.7,4.8){$\overbrace{\hspace{40mm}}^{}$} \put(11.9,5.4){$M_{2}$}  
			\put(3.5,1.2){$\underbrace{\hspace{83mm}}_{}$} \put(8.8,0.3){$M$}  
			\put(9,5.5){$N$}
		\end{picture}
		\caption{}
		\label{fig1}
	\end{figure}

	We then glue $M_1$, $M_2$ and $[-1,1]\times N$ naturally (see Figure \ref{fig2}), we get a new fiberation $\bm\to S$, which is isomorphic to the original fiberation $M\to S.$ Let $T^H\bm,g^{T\bz}$ and $h^{\F}$ be the natural extension of $T^HM,g^{TZ}$ and $h^F$ to $\bm$.
	
	\begin{figure}[h]
		\setlength{\unitlength}{0.75cm}
		\centering
		\begin{picture}(18,6.5)
			\qbezier(1,2)(3,1)(5,2) 
			\qbezier(1,4)(3,5)(5,4) 
			\qbezier(1,2)(-0.5,3)(1,4) 
			\qbezier(2,3.2)(3,2.7)(4,3.2) 
			\qbezier(2.5,3)(3,3.2)(3.5,3) 
			\qbezier(5,2)(5.5,2.3)(6,2.4) 
			\qbezier(5,4)(5.5,3.7)(6,3.6) 
			\qbezier(13,2)(15,1)(17,2) 
			\qbezier(13,4)(15,5)(17,4) 
			\qbezier(17,2)(18.5,3)(17,4) 
			\qbezier(14,3.2)(15,2.7)(16,3.2) 
			\qbezier(14.5,3)(15,3.2)(15.5,3) 
			\qbezier(13,2)(12.5,2.3)(12,2.4) 
			\qbezier(13,4)(12.5,3.7)(12,3.6) 
			\qbezier(6,2.4)(9,2.4)(12,2.4) 
			\qbezier(6,3.6)(9,3.6)(12,3.6) 
			\qbezier(7,2.4)(6.7,3)(7,3.6) 
			\qbezier[30](7,2.4)(7.3,3)(7,3.6) 
			\qbezier(11,2.4)(10.7,3)(11,3.6) 
			\qbezier[30](11,2.4)(11.3,3)(11,3.6) 
			
			
			\put(2.8,2){$M_1$} 
			\put(14.8,2){$M_2$} 
			\put(6.2,2){$\underbrace{\hspace{43mm}}^{}$} \put(8,1){$[-2,2]\times N$}  
			\put(1,5.0){$\overbrace{\hspace{59mm}}^{}$} \put(4.1,5.5){$\bm_1$}  
			\put(7.2,4){$\overbrace{\hspace{27mm}}^{}$} \put(8,4.5){$[-1,1]\times N$}  
			\put(9.2,5.0){$\overbrace{\hspace{58mm}}^{}$} \put(12.9,5.5){$\bm_2$}  
			\put(1.5,0.5){$\underbrace{\hspace{113mm}}_{}$} \put(8.8,-0.5){$\bm$}  
		\end{picture}
		\caption{}
		\label{fig2}
	\end{figure}

	Let $p_T$ be a smooth function on $\bm$ (see Figure \ref{fig3}), such that
	\begin{enumerate}
		\item $p_T|_{M_1}\equiv-T/2,$
		\item $p_T|_{M_2}\equiv T/2$,
		\item $p_T|_{[-1,0]\times Y}(s,y)\approx T(s+1)^2/2-T/2$,
		\item $p_T|_{[0,1]\times Y}(s,y)\approx -T(s-1)^2/2+T/2$.
	\end{enumerate}

	\begin{figure}[h]
		\begin{center}
			\begin{tikzpicture}
				\draw[->,thick] (-4,0) -- (4,0);
				\draw[->,thick] (0,-1.5) -- (0,1.5);
				\draw[domain=-4:-1] plot (\x,{-1});
				\draw[domain=-1:0] plot (\x,{(\x+1)*(\x+1)-1});
				\draw[domain=0:1] plot (\x,{-(\x-1)*(\x-1)+1});
				\draw[domain=1:4] plot (\x,{1});
				\draw [dashed] plot coordinates {(-1,-1)(0,-1)};
				\draw [dashed] plot coordinates {(1,1)(0,1)};
				\node at (0.5,-1) {$-T/2$};
				\node at (-0.5,1) {$T/2$};
				\draw [dashed] plot coordinates {(-1,1.4)(-1,-1.4)};
				\draw [dashed] plot coordinates {(1,1.4)(1,-1.4)};
				\node at (-2.5,0.6) {$M_1$};
				\node at (2.5,0.6) {$M_2$};
				\fill [black] (-1,0) circle (2pt);
				\fill [black] (1,0) circle (2pt);
				\node at (-1.3,0.3){$-1$};
				\node at (1.2,0.3){$1$};
			\end{tikzpicture}
		\end{center}
		\caption{}
		\label{fig3}
	\end{figure}

	For simplicity, we solely describe the situation where $S=\{pt\}$. In this case, the analytic torsion form reduces to the logarithm of analytic torsion.  
	
	Let $d_T:=d^{\F}+dp_T\wedge$, $d_T^{*}$ be the formal adjoint of $d_T$. Set $D_T:=d_T+d_T^{*}$, $\Delta_T:=D_T^2.$ 
	
	Let $\l_k$ be the $k$-th eigenvalue (counted with multiplicities) of $\Delta_1\oplus\Delta_2$ (acting on $\Omega_{\mathrm{rel}}(M_1;F_1)\oplus\Omega_{\mathrm{abs}}(M_2;F_2)$) . Let $\{\l_k(T)\}$ be the $k$-th eigenvalue (counted with multiplicities) of $\Delta_T$ (acting on $\Omega(\bm;\F)$).
	
	One has a nice observation that \be\label{ob1}\lim_{T\to\infty}\l_k(T)=\l_k.\ee
	
	We temporarily assume that $\dim H^k(M)=\dim H^k(M_1)+\dim H^k(M_2,\p M_2)$. Let 
	$$\mathcal{T}(T^H\bm,g^{T\bz},h^{\F})(T)$$ be the analytic torsion with respect to $\Delta_T$. Then based on (\ref{ob1}), naively, one should expect that 
	\be\label{ideq}\lim_{T\to\infty} \mathcal{T}(T^H\bm,g^{T\bz},h^{\F})(T)= \mathcal{T}_{1}(T^HM_1,g^{TZ_1},h^{F_1})+ \mathcal{T}_{2}(T^HM_2,g^{TZ_2},h^{F_2}),\ee and $$\lim_{T\to 0}\mathcal{T}(T^H\bm,g^{T\bz},h^{\F})(T)=\mathcal{T}(T^H\bm,g^{T\bz},h^{\F}).$$ As a result, the relationship between analytic torsion on $M$ and analytic torsion on the two pieces above could be seen, proving Theorem \ref{main0}. 
	
	Now, we no longer assume that  $\dim H^k(M)=\dim H^k(M_1)+\dim H^k(M_2,\p M_2)$. Note that for large $T$, the bundle $\Omega_{\sm}(M,F)(T)$ generated by the eigenforms of $\Delta_T$ with respect to small eigenvalues is of finite rank. The second key ingredient in our proof is the relationship between the analytic torsion form for the exact sequence (\ref{exactse}) and $\Omega _{\sm}(M,F)(T)$ as $T\to\infty$. See \cref{lastreal} for more details.

	\subsection{Organization}
	In \cref{section1}, we will give a brief review of analytic torsion forms and establish the basic settings. In \cref{intr}, we state and prove several intermediate results, then prove Theorem \ref{main0}. While Theorem \ref{eigencon}, \ref{limitfhat}, \ref{larcon}, \ref{main} and \ref{int6p} will be proved in subsequent sections. In \cref{eigen}, we investigate the behavior of eigenvalues as $T\to\infty$ and prove Theorem  \ref{eigencon}, \ref{limitfhat} and \ref{larcon}.  In \cref{cons}, we establish Theorem \ref{main}. Theorem \ref{int6p} is proved in \cref{lastreal}.

	\section{Preliminary}\label{section1}
	Let $\pi: M\rightarrow S$ be a fibration of $C^{\infty}$ manifolds with closed fibre $Z$. Let $T Z\subset TM$ be the vertical tangent bundle of the fiber bundle. Let $F$ be a flat complex vector bundle on $M$ with Hermitian metric $h^F$ and a flat connection $\nabla^F$. For $\s\in S,$ $Z_\s:=\pi^{-1}(\s)$ and $F_\s:=F|_{Z_\s}.$

	From now on, we assume that $S$ is closed, and $T^HM,g^{TZ}$ and $h^F$ are product-type near $N$. That is, there exists a neighborhood $U$ of $N$, such that $U= (-1,1)\times N$, and let $(s,z)$ be its coordinate. Then $T^HU=\{0\}\times T^HN$ for some splitting $TN=T^HN\oplus TY$, $g^{TZ}|_{U}=ds\otimes ds+g^{TY}$ for some metric $g^{TY}.$ Let $h^F_N:=h^F|_{\{0\}\times N}$. For any $v\in F_{(s,z)}$, let $P_{\gamma}\in \End(F_{(s,z)},F_{(0,z)})$ be the parallel transport associated with $\nabla^F$ w.r.t. the path $\gamma(t)=(st,z),t\in[0,1]$, then we require that $h^F(v,v)=h^F_N(P_\gamma v,P_\gamma v).$ Thus, $\nabla_{\frac{\p}{\p s}}^Fh^F=0$ in $U.$
	
	If $T$ is sufficiently large, all constants appearing in this paper are at least independent of $T$ and $\s$. The notations $C$ and $c$, et cetera, denote constants that may vary based on context.

	\subsection{Definition of Bismut-Lott analytic torsion forms}\label{defnbl}
	Let $T^H M$ be a sub-bundle of $T M$ such that
	$$
	T M=T^H M \oplus T Z .
	$$
	Let $P^{T Z}$ denote the projection from $T M$ to $T Z$ w.r.t. the decomposition above. If $U \in T S$, let $U^H$ be the lift of $U$ in $T^H M$, s.t. $\pi_* U^H=U$.
	
	For a Hilbert bundle $H\to X$, $\Gamma(X,H)$ denotes the space of smooth sections. 
	
	Now let $E^i=\Omega^i(Z,F)$, which is the bundle over $S$, such that for $\s\in S$, $\Omega^i(Z,F)|_{\Zs}=\Omega^i(\Zs,F_\s),$ that is, $E=\oplus_{i=0}^{\operatorname{dim} Z} E^i$  is the smooth infinite-dimensional $\Z$-graded vector bundle over $S$, whose fiber at $\s \in S$ is $\Gamma\left(Z_\s,\left(\Lambda\left(T^* Z\right) \otimes F\right)|_{Z_\s}\right)$. Then
	$$
	\Gamma\left(S ,\Omega^i(Z,F)\right)=\Gamma\left(M, \Lambda^i\left(T^* Z\right) \otimes F\right).
	$$
	
	For $s \in \Gamma(S ; E)$ and $U\in \Gamma(S,TS)$, the Lie differential $L_{U^H}$ acts on $\Gamma(S, E)$. Set
	$$
	\nabla_U^E s=L_{U^H} s .
	$$
	Then $\nabla^E$ is a connection on $E$ preserving the $\mathbb{Z}$-grading.
	
	If $U_1, U_2$ are vector fields on $S$, put
	$$
	T\left(U_1, U_2\right)=-P^{T Z}\left[U_1^H, U_2^H\right] \in \Gamma(M, T Z) .
	$$
	We denote $i_T \in \Omega^2\left(S, \operatorname{Hom}\left(E^{\bullet}, E^{\bullet-1}\right)\right)$ be the 2 -form on $S$ which, to vector fields $U_1, U_2$ on $S$, assigns the operation of interior multiplication by $T\left(U_1, U_2\right)$ on $E$. 
	Let $d^{Z}:\Omega^*(Z,F)\to\Omega^{*+1}(Z,F)$ be exterior differentiation along fibers induced by $\nabla^F$. We consider $d^Z$ to be an element of $\Gamma\left(S ; \operatorname{Hom}\left(E^{\bullet}, E^{\bullet+1}\right)\right)$. By \cite[Proposition 3.4]{bismut1995flat}, we have
	$$
	d^M=d^Z+\nabla^E+i_T .
	$$
	Here $d^M:\Omega^*(M,F)\to \Omega^{*+1}(M,F)$ is induced by $\nabla^F.$
	So $d^M$ is a flat superconnection of total degree 1 on $E$. $(d^M)^2=0$ implies that
	$$
	\left(d^Z\right)^2=0,[\nabla^E, d^Z]=0.
	$$
	Let $g^{T Z}$ be a metric on $T Z$. Let $h^E$ be the metric on $E$ induced by $h^F$ and $g^{TZ}$. Sometimes we will also denote $h^E$  by $(\cdot,\cdot)_{L^2(Z)}$.
	
	Let $\nabla^{E,*}, d^{Z ,*},d^{M,*}$ be the formal adjoint of $\nabla^E, d^Z, d^M$ with respect to  ${h^E}$. Set
	$$
	D^Z=d^Z+d^{Z *}, \quad \nabla^{E, u}=\frac{1}{2}\left(\nabla^E+\left(\nabla^E\right)^*\right) .
	$$
	Let $N^Z$ be the number operator of $E$, i.e. acts by multiplication by $k$ on the space $\Gamma\left(W, \Lambda^k\left(T^* Z\right) \otimes F\right)$. For $t>0$, set
	$$
	\begin{array}{l}
		C_t^{\prime}=t^{N^Z / 2} d^M t^{-N^Z / 2}, \quad C_t^{\prime \prime}=t^{-N^Z / 2}\left(d^M\right)^* t^{N^Z / 2}, \\
		C_t=\frac{1}{2}\left(C_t^{\prime}+C_t^{\prime \prime}\right), \quad D_t=\frac{1}{2}\left(C_t^{\prime \prime}-C_t^{\prime}\right) ,
	\end{array}
	$$
	then $C_t^{\prime \prime}$ is the adjoint of $C_t^{\prime}$ with respect to $h^E$ . $C_t$ is a superconnection and $D_t$ is an odd element of $\Omega(S, \End (E))$, and
	$$
	C_t^2=-D_t^2.
	$$
	For $X \in T Z$, let $X^* \in T^* Z$ correspond to $X$ by the metric $g^{T Z}$. Set $c(X)=X^* \wedge-i_X$. Then
	$$
	C_t=\frac{\sqrt{t}}{2} D^Z+\nabla^{E, u}-\frac{1}{2 \sqrt{t}} c(T).
	$$
	Let $f (a) = a \exp(a^2)$, and $\varphi: \Omega^{\text {even }}(S) \rightarrow \Omega^{\text {even }}(S)$ as follows,
	$$
	\varphi \omega=(2 \pi i)^{-k} \omega, \quad \text { for } \omega \in \Omega^{2 k}(S) \text {. }
	$$

	For any $t>0$, the operator $D_t$ is a fiberwise-elliptic differential operator. Then $f\left(D_t\right)$ is a fiberwise trace class operator. For $t>0$, put
	$$
	f^{\wedge}\left(C_t^{\prime}, h^E\right):=\varphi \operatorname{Tr}_s\left(\frac{N^Z}{2} f^{\prime}\left(D_t\right)\right).
	$$

	Put
	$$
	\begin{array}{l}
		\chi(Z,F):=\sum_{j=0}^{\operatorname{dim} Z}(-1)^j \operatorname{rank} H^j(Z, F), \\
		\chi^{\prime}(Z, F):=\sum_{j=0}^{\operatorname{dim} Z}(-1)^j j \operatorname{rank} H^j\left(Z, F\right).
	\end{array}
	$$
	
	\begin{defn}\label{torc}
		The analytic torsion form $\mathcal{T}\left(T^H M, g^{T Z}, h^F\right)$ is a form on $S$ which is given by
		\begin{align*}
			&\ \ \ \ \mathcal{T}\left(T^H M, g^{T Z}, h^F\right)\\
			&=-\int_0^{+\infty}\left(f^{\wedge}\left(C_t^{\prime}, h^E\right)-\frac{\chi^{\prime}(Z, F)}{2} 
			-\frac{  \chi(Z,F)\dim Z-2 \chi^{\prime}(Z, F)}{4} f^{\prime}\left(\frac{\sqrt{-1} \sqrt{t}}{2}\right)\right) \frac{d t}{t} .
		\end{align*}
	\end{defn}
	
	It follows from \cite[Theorem 3.21]{bismut1995flat} that $\mathcal{T}\left(T^H M, g^{T Z}, h^F\right)$ is well defined. The degree $0$ part of $\mathcal{T}\left(T^H M, g^{T Z}, h^F\right)$ is nothing but the fiberwise analytic torsions. This is why $\mathcal{T}\left(T^H M, g^{T Z}, h^F\right)$ is referred to as analytic torsion forms.
	
	\def\Ho{\mathrm{Hodge}}

	\subsection{Analytic torsion forms for manifolds with boundary}\label{defnblbd}
	Let $N \subseteq M$ be a hypersurface transversal to $Z$. We suppose that $\pi|_N: N \rightarrow S$ is  a fibration of fiber $Y:=N \cap Z$. Suppose that $N$ cuts $M$ into two pieces $M_1$ and $M_2$, and fiberwisely, cut $Z$ into two pieces $Z_1$ and $Z_2$. Let $\pi_i: M_i \rightarrow S$ be the restriction of $\pi$ to $M_i$. Then $\pi_i$ is a fibration of $Z_i$ ($i=1,2$). Let $F_i, T^HM_i$ and $h^{F_i}$ be the restriction of $F, T^HM$ and $h^{F}$ to $M_i$ respectively ($i=1,2$). First, identify a neighborhood $U_1$ of $\p M_1$ with $(-2,-1]\times N$, and identify $\p M_1$ with $\{-1\}\times N$; Then identify a neighborhood  $U_2$ of $\p M_2$ with $[1,2)\times N$, and identify $\p M_2$ with $\{1\}\times N$. Let $(s,z)$ be the coordinate of $U_i$ w.r.t. the identification above. Moreover, assume that $\pi_i|_{U_i}(s,z)=\pi|_N(z)$. 
	
	Let $\Omega(Z_i,F)$ denotes the space of $F|_{Z_i}$-valued smooth differential forms, and
	$$\begin{aligned} 
		\Omega_{\mathrm{abs}}\left(Z_1, F\right) &=\left\{\omega \in \Omega\left(Z_1, F\right):\left.i_{\frac{\p}{\p s}} \omega\right|_{\partial Z_1}=\left.i_{\frac{\p}{\p s}}\left(d^{Z_1} \omega\right)\right|_{\partial Z_1}=0\right\}, \\
		\Omega_{\mathrm{rel}}\left(Z_2, F\right) &=\left\{\omega \in \Omega\left(Z_2, F\right):\left.de\wedge\omega\right|_{\partial Z_2}=\left.ds\wedge\left(d^{Z_2,*} \omega\right)\right|_{\partial Z_2}=0\right\}.
	\end{aligned}$$
	We write $E_i=\Omega_{\mathrm{b d}} (Z_i;F)$ for short if the choice of abs/rel is clear. 
	
	Let $D^{Z_i}$ be the restriction of $D^{Z}$ on $Z_i$ acting on $\Omega_{\mathrm{b d}} (Z_i;F_i)$. Similarly, we have $d^{M_i},\nabla^{E_i},C_{t,i},D_{t,i}$ e.t.c.
	
	For $t>0$, put
	$$
	f^{\wedge}\left(C_{t,i}^{\prime}, h^{E_i}\right):=\varphi \operatorname{Tr}_s\left(\frac{N^Z}{2} f^{\prime}\left(D_{t,i}\right)\right),
	$$
	where $h^{E_i}=(\cdot,\cdot)_{L^2(Z_i)}$ is the metric induced by $g^{TZ_i}$ and $h^{F_i}.$

	Put
	$$
	\begin{array}{l}
		\chi_{\bd}(Z_i,F):=\sum_{j=0}^{\operatorname{dim} Z}(-1)^j \operatorname{rank} H_{\bd}^j(Z_i, F_i), \\
		\chi^{\prime}(Z_i, F_i):=\sum_{j=0}^{\operatorname{dim} Z}(-1)^j j \operatorname{rank} H_{\bd}^j\left(Z_i, F_i\right).
	\end{array}
	$$
	
	\begin{defn}\label{torc1}
		The analytic torsion form $\mathcal{T}_i\left(T^H M_i, g^{T Z_i}, h^{F_i}\right)(i=1,2)$ is a form on $S$ which is given by
		\begin{align*}
			&\ \ \ \ \mathcal{T}_i\left(T^H M_i, g^{T Z_i}, h^{F_i}\right)\\
			&=-\int_0^{+\infty}\left(f^{\wedge}\left(C_{t,i}^{\prime}, h^{E_i}\right)-\frac{\chi^{\prime}(Z_i, F_i)}{2} 
			-\frac{  \chi(Z_i,F_i)\dim Z-2 \chi^{\prime}(Z_i, F_i)}{4} f^{\prime}\left(\frac{\sqrt{-1} \sqrt{t}}{2}\right)\right) \frac{d t}{t} .
		\end{align*}
	\end{defn}
	
	It follows from \cite[Theorem 2.17]{zhu2015gluing} that $\mathcal{T}_i\left(T^H M_i, g^{T Z_i}, h^{F_i}\right)$ is well defined.

	\subsection{Analytic torsion form for Witten deformations}\label{aw}
	Let $N \subseteq M$ be a hypersurface transversal to $Z$. We suppose that $\pi|_N: N \rightarrow S$ is  a fibration of fiber $Y:=N \cap Z$. Suppose that $N$ cuts $M$ into two pieces $M_1$ and $M_2$, and fiberwisely, cut $Z$ into two pieces $Z_1$ and $Z_2$. Let $\pi_i: M_i \rightarrow S$ be the restriction of $\pi$ to $M_i$. Then $\pi_i$ is a fibration of $Z_i$ ($i=1,2$). Let $F_i, T^HM_i$ and $h^{F_i}$ be the restriction of $F, T^HM$ and $h^{F}$ to $M_i$ respectively ($i=1,2$). First, identify a neighborhood $U_1$ of $\p M_1$ with $(-2,-1]\times N$, and identify $\p M_1$ with $\{-1\}\times N$; Then identify a neighborhood  $U_2$ of $\p M_2$ with $[1,2)\times N$, and identify $\p M_2$ with $\{1\}\times N$. Let $(s,z)$ be the coordinate of $U_i$ w.r.t. the identification above. 
	
	Let $\bm=M_1\cup[-1,1]\times N\cup M_2$, and $\bar{F},h^{\bar{F}},g^{T\bar{Z}} $ and $T^H\bar{M}$ be the natural extensions of $F,h^F,g^{TZ}$ and $T^HM$ to $\bm$. We still have a natural extension of fiberation $\bm\to S.$ Similar, we have notation $\bar{Z}$ for the fiber $Z.$

	\def\Cn{\mathcal{C}}
	\def\cn{\delta}
	
	Let $p_T$ be a family of odd smooth functions on $[-2,2]$, such that 
	\begin{enumerate}[(a)]
		\item $p_T|_{[1,2]}\equiv T/2$,
		\item $p_T|_{[1/16,1]}(s)=-T\rho(e^{T^2}(1-s))(s-1)^2/2+T/2$ , where $\rho\in C_c^\infty([0,\infty))$, such that $0\leq\rho\leq1$, $\rho_{[0,1/2]}\equiv0,$ $\rho_{[3/4,\infty]}\equiv1,$ $|\rho'|\leq \cn_1$ and $|\rho''|\leq \cn_2$ for some universal constant $\cn_1$ and $\cn_2,$
		\item $\Cn_1T\leq|p_T'|(s)\leq 2\Cn_1 T$, $|p_T''|\leq \Cn_2T$ for some universal constants $\Cn_1$ and $
		\Cn_2$ whenever $s\in[0,1/16].$
		
	\end{enumerate}
	Then one can see that $|p_T'|(s)\leq \Cn_3T||s|-1|$ and $|p_T''|(s)\leq \Cn_4T$ whenever $||s|-1|\leq e^{-T^2}$ for some universal constant $\Cn_3$ and $\Cn_4$. \\
	
	We could think $p_T$ as a function on $\bm.$ Still denotes $p_T$ to be its fiberwise restriction.  
	Let $d^{\bz}_{T}:=d^{\bar{Z}}+dp_T\wedge,$ $d_T^{\bz,*}$ be the adjoint of $d^{\bz}_T$. Then $D_T^{\bz}:=d^{\bz}_T+d_T^{\bz,*},\Delta_T:=(D_T^{\bz})^2.$ Similarly, we have notation $d^{\bm}_T$, $\nabla^{\bE}$, $C_{t,T}$, $D_{t,T}$ e.t.c.
	
	For $t>0$, put
	$$
	f^{\wedge}\left(C_{t,T}^{\prime}, h^{\bE}\right):=\varphi \operatorname{Tr}_s\left(\frac{N^Z}{2} f^{\prime}\left(D_{t,T}\right)\right),
	$$
	where $\bE=\Omega(\bz,\F)$, and $h^{\bE}=(\cdot,\cdot)_{L^2(\bz)}$ is the metric on $\Omega(\bz,\F)$ induced by $g^{T\bz}$ and $h^{\F}.$
	Let $H(\bz;\F)(T)$ be the bundle on $S$, whose fiber at $\s\in S$ is the cohomology $H(Z_\s;F_\s)(T)$ with respect to $d^\bz_T.$

	\begin{defn}\label{torc2}
		The analytic torsion form $\mathcal{T}\left(T^H \bm, g^{T \bz}, h^{\F}\right)(T)$ is a form on $S$ which is given by
		\begin{align*}
			&\ \ \ \ \mathcal{T}\left(T^H \bm, g^{T \bz}, h^{\F}\right)(T)\\
			&=-\int_0^{+\infty}\left(f^{\wedge}\left(C_{t,T}^{\prime}, h^{\bar{E}}\right)-\frac{\chi^{\prime}(\bz, \F)}{2} 
			-\frac{  \chi(\bz,\F)\dim Z-2 \chi^{\prime}(\bz, \F)}{4} f^{\prime}\left(\frac{\sqrt{-1} \sqrt{t}}{2}\right)\right) \frac{d t}{t} .
		\end{align*}
		
	\end{defn}
	It follows from \cite[Theorem 3.21]{bismut1995flat} and discussions in \cref{witwei} that $\mathcal{T}\left(T^H M, g^{T Z}, h^F\right)$ is well defined for a fixed $T>0.$
	
	Lastly, for the sake of convenience, $(\cdot,\cdot)_{L^2}$ (resp. $\|\cdot\|_{L^2}:=\sqrt{(\cdot,\cdot)_{L^2}}$) will be adopted to represent $(\cdot,\cdot)_{L^2(Z)}$ (resp. $\|\cdot\|_{L^2(Z)}:=\sqrt{(\cdot,\cdot)_{L^2(Z)}}$) , $(\cdot,\cdot)_{L^2(\bz)}$ (resp. $\|\cdot\|_{L^2(\bz)}:=\sqrt{(\cdot,\cdot)_{L^2(\bz)}}$)  or $(\cdot,\cdot)_{L^2(Z_i)}$(resp. $\|\cdot\|_{L^2(Z_i)}:=\sqrt{(\cdot,\cdot)_{L^2(Z_i)}}$)  ($i=1,2$), when the context is clear.
	\subsubsection{Witten Laplacian v.s. weighted Laplacian}\label{witwei}
	
	Instead of deforming the de Rham differential $d^{\bz}$, we could also deform the metric $h^{\F}$: let $h^{\F}_T:=e^{-2p_T}h^{\F}$. Similarly, $g^{T\bz}$ and $h^{\F}_T$ induce an $L^2$-norm $h_T^{\bE}=(\cdot,\cdot)_{L^2(\bz),T}$ on $\bE=\Omega(\bz;\F).$
	
	Then the formal adjoint $\delta^{\bz,*}_T$ of $d^{\bz}$ w.r.t. the $(\cdot,\cdot)_{L^2,T}$ is then given by $e^{p_T}d_T^{\bz,*}e^{-p_T}$. Then $\tilde{D}^{\bz}_T:=d^{\bz}+\delta_T^{\bz,*},\tD_T:=(\tilde{D}^{\bz}_T)^2.$ Similarly, we have notation $\tilde{d}^{\bm}_T$, $\tilde{C}_{t,T}$, $\tilde{D}_{t,T}$ e.t.c.

	The Weighted Laplacian $\tD_T$ is given by $\tD_T=d^{\bz}\delta^{\bz,*}_T+\de^{\bz,*}_Td^{\bz},$ one can see that $\tD_T=e^{p_T}\Delta_T e^{-p_T}.$ Let $l_k(T)$ be the $k$-th eigenvalue of $\tD_T$, then $l_k(T)=\l_k(T).$ Moreover, if $u$ is an eigenform of $\Delta_T$ w.r.t. eigenvalue $\l$, then $e^{p_T}u$ is an eigenform of $\tD_T$ w.r.t. eigenvalue $\l.$ 
	
	As a result, $f^{\wedge}\left(C_{t,T}^{\prime}, h^{\bE}\right)=f^{\wedge}\left(\tilde{C}_{t,T}^{\prime}, h_T^{\bE}\right)$.
	
	\subsubsection{Absolute/Relative boundary conditions for weighted Laplacian}
	Let $\bm_1:=M_1\cup [-1,0]\times N$, $\bm_2:=M_2\cup[0,1]\times N$, and $\bz_1:=Z_1\cup [-1,0]\times Y$, $\bz_2:=Z_2\cup[0,1]\times Y$ (see Figure \ref{fig2}). Let $\F_i$ be the restriction of $\F$ on $\bm_i$ ($i=1,2$).
	Set
	\begin{align*}
		\Omega_{\mathrm{abs}} (\bz_1;\F_1)&:=\left\{\omega \in \Omega (\bz_1;\F_1): i_{\frac{\partial}{\partial s}} \omega=0, i_{\frac{\partial}{\partial s}} d^{\bz_1} \omega=0\mbox{ on }\{0\}\times Y\right\},\\
		\Omega_{\mathrm{rel}} (\bz_2;\F_2)_T&:=\left\{\omega \in \Omega (\bz_2;\F_2): d s \wedge \omega=0, d s \wedge \delta_T^{\bz_2, *} \omega=0 \text { on } \{0\}\times Y\right\}.
	\end{align*}

	Let $\tD_{T,i}$ be the restriction of $\tD_T$ acting on $\Omega_{\bd}(\bz_i;\F_i)$. Then by Hodge theory, $\ker(\tD_{T,i})\cong H_{\bd}(\bz_i;\F_i)$.  Similarly, we have notation $\tilde{d}^{\bm}_{T,i}$, $\tilde{C}_{t,T,i}$, $\tilde{D}_{t,T,i}$ e.t.c.
	
	\begin{defn}\label{torc4}
		The analytic torsion form $\mathcal{T}_i\left(T^H \bm_i, g^{T \bz_i}, h_T^{\F_i}\right)(T)$ is a form on $S$ which is given by
		\begin{align*}
			&\ \ \ \ \mathcal{T}_i\left(T^H \bm_i, g^{T \bz_i}, h_T^{\F_i}\right)(T)\\
			&=-\int_0^{\infty}\left(f^{\wedge}\left(\tC_{t,T'}^{\prime}, h_T^{\bE_i}\right)-\frac{\chi^{\prime}(Z_i, F_i)}{2} 
			-\frac{  \chi(Z_i,F_i)\dim Z-2 \chi^{\prime}(Z_i, F_i)}{4} f^{\prime}\left(\frac{\sqrt{-1} \sqrt{t}}{2}\right)\right) \frac{d t}{t} .
		\end{align*}
	\end{defn}
	
	Lastly, $g^{\bz_i}$ and $h^{\F_i}_T$ induce an $L^2$-norm $(\cdot,\cdot)_{L^2(\bz_i),T}$ on $\Omega_{\bd}(\bz_i;\F_i).$
	
	For the sake of convenience, $(\cdot,\cdot)_{L^2,T}$ (resp. $\|\cdot\|_{L^2,T}:=\sqrt{(\cdot,\cdot)_{L^2,T}}$)  will be adopted to represent $(\cdot,\cdot)_{L^2(\bz),T}$ (resp. $\|\cdot\|_{L^2(\bz),T}:=\sqrt{(\cdot,\cdot)_{L^2(\bz),T}}$)  or $(\cdot,\cdot)_{L^2(\bz_i),T} $ (resp. $\|\cdot\|_{L^2(\bz_i),T}:=\sqrt{(\cdot,\cdot)_{L^2(\bz_i),T}}$)  ($i=1,2$), when the context is clear.

	\subsection{Analytic torsion form for  complex of finite dimensional vector bundles}\label{defnblfn}
	
	Let $X$ be a closed manifold. Let
	$$
	(E, \nu): 0 \rightarrow E^0 \stackrel{\nu}{\rightarrow} E^1 \stackrel{\nu}{\rightarrow} \cdots \stackrel{\nu}{\rightarrow} E^n \rightarrow 0 .
	$$
	be a flat complex of complex vector bundles on $X$. That is $\nabla^E=\oplus_{i=0}^n \nabla^{E^i}$ is a flat connection on $E=\oplus_{i=0}^n E^i$ and $\nu$ is a flat chain map, meaning by
	$$
	\left(\nabla^E\right)^2=0, \nu^2=0, \nabla^E \nu=0 .
	$$
	Then $\nu+\nabla^E$ is a flat superconnection of total degree $1$ . By \cite[\S 2(a)]{bismut1995flat}, the cohomology $H(E, v)$ of the complex is a vector bundle on $X$, and let $\nabla^{H(E, v)}$ be the flat connection on $H(E, v)$ induced by $\nabla^E$. Let
	$$
	\begin{array}{l}
		d(E)=\sum_{j=0}^n(-1)^j j \mathrm{rank} E^j ,\\
		d(H(E, v))=\sum_{j=0}^n(-1)^j j \mathrm{rank} H^j(E, v).
	\end{array}
	$$
	
	Let $h^E=\oplus h^{E_i}$ be a metric on $E=\oplus E^i$. Let $\nu^*$ and $\nabla^{E,*}$ be the formal adjoint of $\nu$ and $\nabla^E$ with respect to $h^E$. Let $N$ be the number operator on $E$, i.e. $N$ acts by multiplication by $i$ on $E^i$. Set
	$$
	\begin{array}{l}
		f\left(\nabla^E, h^E\right)=\sum_{i=0}^n(-1)^i f\left(\nabla^{E^i}, h^{E^i}\right) \\
		f\left(\nabla^{H(E, v)}, h^{H(E, v)}\right)=\sum_{i=0}^n(-1)^i f\left(\nabla^{H^i(E, v)}, h^{H(E, v)}\right)
	\end{array}
	$$
	For $t>0$, let
	$$
	D_t=\frac{1}{2} \sqrt{t}\left(v^*-v\right)+\frac{1}{2} (\nabla^{E,*}-\nabla^E).
	$$
	The analytic torsion form for the complex of finite dimensional vector bundles is defined as
	\begin{defn}
		$$\begin{aligned} &\ \ \ \ \T_f\left(\nu+\nabla^E, h^E\right)\\
			&=-\int_0^{\infty}\left(\varphi \operatorname{Tr}_s\left(\frac{1}{2} N f^{\prime}\left(D_t\right)\right)-\frac{1}{2} d(H(E, v))-\half\left(d(E)-d(H(E, v)\right) f^{\prime}\left(\frac{\sqrt{-1} \sqrt{t}}{2}\right)\right) \frac{d t}{t} .\end{aligned}$$
	\end{defn}

	\section{Intermidiate Results}\label{intr}
	In this section, we will state and prove some intermediate results to prove Theorem \ref{main0}.
	For each $\s\in S$, denote $D^{\bz}_{T}(\s)$ (resp. $D^{Z_i}(\s)$ and $\tilde{D}^{\bz_i}_{T}(\s)$) to be the restriction of $D^{\bz}_T$ (resp. $D^{Z_i}$ and $\tilde{D}^{\bz}_T$) on $\bzs$ (resp. $Z_{i,\s}$ and $\bz_{i,\s}$), $\Delta_T(\s):=(D^{\bz}_T(\s))^2$ (resp. $\Delta_{i}(\s):=(D^{Z}_{i}(\s))^2$ and $\tilde\Delta_{T,i}(\s)$). Here $Z_{i,\s}:=\pi_i^{-1}(\s)$ ($i=1,2$).
	
	Let $\l_k(T,\s)$ be the $k$-th eigenvalue of $\Delta_T(\s)$, $\l_k(\s)$ be the $k$-th eigenvalue of $\Delta_1(\s)\oplus\Delta_2(\s)$ acting on $\Omega_{\abs}(\Zso;F_{1,\s})\oplus\Omega_{\rel}(\Zsw;F_{2,\s}),$ and $\tl_k(T,\s)$ be the $k$-th eigenvalue of $\tilde\Delta_{T,1}(\s)\oplus\tilde\Delta_{T,2}(\s)$. 
	
	Moreover, to avoid heavy notation, we will denote $\bz,\F$, $D^{Z_i},\Delta_T$  et cetera for $\bz_{\s},\F_\s, D^{Z_i}_\s,\Delta_T(\s)$ et cetera,  if there is no need to specify the base point $\s$. 
	
	\begin{thm}\label{eigencon}
		$\lim_{T\to\infty}\l_k(T,\s)=\lim_{T\to\infty}\tl_k(T,\s)=\l_k(\s)$ uniformly in $S$. That is, for example, for every $\epsilon>0$, there exists $T_k=T_k(\epsilon)>0$ that doesn't depend on $\s$, such that whenever $T\geq T_k$, $|\l_k(T,\s)-\l_k(\s)|<\epsilon.$
	\end{thm}

	By Hodge theory, there exists $k_0\in\Z$, such that whenever $k< k_0$ $\l_k(\theta)\equiv0$, $\l_{k_0}(\theta)\neq0$. Let $\delta:=\frac{1}{2}\inf_{\s\in S}\l_{k_0}(\s)>0.$ Then by Theorem \ref{eigencon}, when $T$ is large enough, all eigenvalues of $\Delta_T$ inside $[0,\delta]$ converge to $0$ as $T\to\infty$. 
	
	Let $\Omega_{\sm}(\bz,\bar{F})(T)$ be the vector bundle over $S$, such that for all $\s\in S$, $\Omega_{\sm}(\bz_{\s},\F_\s)(T)$ is the space generated by eigenforms of $\Delta_T$ for eigenvalues inside $[0,\delta]$,
	and $\P^{\delta}(T)$ be the orthogonal projection w.r.t. $\Omega_{\sm}(\bz,\bar{F})(T)$.
	
	On $\Omega_{\sm}(\bz,\bar{F})(T)$, let  $\nabla^{\delta,T}:=\P^{\delta}(T)\nabla^{\bE}$ and $D_t^{\delta,T}:=\nabla^{\delta,T}-\nabla^{\delta,T,*}-\half\sqrt{t}(d^{\bz}_T-d^{\bz,*}_T)$.
	For $t>0$, put
	$$
	f_{\la}^{\wedge}\left(C_{t,T}^{\prime}, h^E\right):=\varphi \operatorname{Tr}_s\left(\frac{N^Z}{2} f^{\prime}\left(D_{t,T}\right)\right)-\varphi \operatorname{Tr}_s\left(\frac{N^Z}{2} \P^{\delta}(T)f^{\prime}\left(\P^{\delta}(T)D^{\delta,T}_t\P^{\delta}(T)\right)\P^{\delta}(T)\right),
	$$
	$$
	f_{\sm}^{\wedge}\left(C_{t,T}^{\prime}, h^E\right):=\varphi \operatorname{Tr}_s\left(\frac{N^Z}{2} \P^{\delta}(T)f^{\prime}\left(\P^{\delta}(T)D_{t}^{\delta,T}\P^{\delta}(T)\right)\P^{\delta}(T)\right).
	$$
	Then proceeding as in the proof of \cite[Theorem 2.13]{bismut1995flat} (simply replace $\frac{\P^{\ker(V)}}{\l}$ in \cite[(2.46)]{bismut1995flat} by $\P^{\delta}(T)(\l-\sqrt{t}(d^{\bz}_T-d^{\bz,*}_T))^{-1}$ e.t.c.),
	\begin{thm}\label{intnew1}For some $\s$-independent constants $C(T),C'(T)>0$, such that for $t\geq1$
		\[\left|f_{\la}^{\wedge}\left(C_{t,T}^{\prime}, h^E\right)\right|\leq \frac{C(T)}{\sqrt{t}}.\]
		For $t\in(0,1]$,
		\begin{align*}
			\left|f_{\la}^{\wedge}\left(C_{t,T}^{\prime}, h^E\right)-\frac{  \chi(\bz,\F)\dim(Z)-2d(\Omega_{\sm}(\bm,\F)(T))}{4}\right|\leq C'(T)t.
		\end{align*}
		Here we put a metric $g^{TS}$ on $S$, and for $\a\in\Omega(S),$ $|\a|:=\sqrt{g^{TS}(\a,\a)}.$
	\end{thm}
	\begin{rem}\label{33}
		The key challenge in this article is the possible $T$-dependence of the constants $C(T)$ and $C'(T)$. Similar issues can be addressed by introducing a two-parameter deformation and taking the adiabatic limit of analytic torsion forms, as is done in \cite{puchol2020adiabatic}. However, we use a different approach in this paper. We figure out that when $t\in[1,\infty)$,
		$\left|f_{\la}^{\wedge}\left(C_{t,T}^{\prime}, h^E\right)\right|$ could be bounded by a $(T,\s)$-independent measurable function $G(t)$, such that $t^{-1}G(t)$ is in $L^{1}([1,\infty))$; when $t\in[0,1]$, the positive degree component of $f_{\la}^{\wedge}\left(C_{t,T}^{\prime}, h^E\right)$ could be bounded by $C't$ for some $(T,\s)$-independent $C'$, while the degree $0$ component of $f_{\la}^{\wedge}\left(C_{t,T}^{\prime}, h^E\right)$ is related to the analytic torsion. To deal with the degree $0$ component, rather than the adiabatic limit used in \cite{10.2140/apde.2021.14.77}, a coupling technique is introduced in \cite[\S 7.1]{Yantorsions}. Together with Theorem \ref{eigencon} and dominated convergence theorem, one can understand $\mathcal{T}_{\la}\left(T^H \bm, g^{T \bz}, h^{\F}\right)(T)$ as $T\to\infty.$
	\end{rem} 
	By Theorem \ref{intnew1}, we could set
	\begin{align*}
		&\ \ \ \ \mathcal{T}_{\la}^{\mL}\left(T^H \bm, g^{T \bz}, h^{\F}\right)(T)\\
		&=-\int_1^{\infty}\left(f_{\la}^{\wedge}\left(C_{t,T}^{\prime}, h^{\bar{E}}\right)- \frac{  \chi(\bz,\F)\dim(Z)-2d(\Omega_{\sm}(\bm,\F)(T))}{4} f^{\prime}\left(\frac{\sqrt{-1} \sqrt{t}}{2}\right)\right)\frac{d t}{t},
	\end{align*}
	\begin{align*}
		&\ \ \ \ \mathcal{T}_{\la}^{\mS}\left(T^H \bm, g^{T \bz}, h^{\F}\right)(T)\\
		&=-\int_0^{1}\left(f_{\la}^{\wedge}\left(C_{t,T}^{\prime}, h^{\bar{E}}\right)
		-\frac{  \chi(\bz,\F)\dim(Z)-2d(\Omega_{\sm}(\bm,\F)(T))}{4} f^{\prime}\left(\frac{\sqrt{-1}\sqrt{t}}{2}\right)\right) \frac{d t}{t} .
	\end{align*}
	Similarly, let $\P_i$ be the be the orthogonal projection w.r.t. $\ker(\Delta_i)$ $i=1,2$,
	and $\tP_i(T)$ be the be the orthogonal projection w.r.t. $\ker(\tilde\Delta_{T,i})$ $i=1,2$.
	
	Let $\nabla^{\H_i}:=\P_i\nabla^{E_i}$ and $D^{\H_i}_t:=\nabla^{\H_i}-\nabla^{\H_i,*}+\half\sqrt{t}(d^{Z_i}-d^{Z_i,*})$.
	
	For $t>0$, put
	\begin{align*}
		&\ \ \ \ f_{\la}^{\wedge}\left(C_{t,i}^{\prime}, h^E\right):=\varphi \operatorname{Tr}_s\left(\frac{N^Z}{2} f^{\prime}\left(D_{t,i}\right)\right)-\varphi \operatorname{Tr}_s\left(\frac{N^Z}{2} \P_if^{\prime}\left(\P_iD^{\H_i}_{t}\P_i\right)\P_i\right)\\
		&=\varphi \operatorname{Tr}_s\left(\frac{N^Z}{2} f^{\prime}\left(D_{t,i}\right)\right)-\frac{\chi^{\prime}(Z_i, F_i)}{2} \mbox{ (By \cite[Theorem 3.15]{bismut1995flat})}.
	\end{align*}
	Similarly,
	\begin{align*}
		&\ \ \ \ f_{\la}^{\wedge}\left(\tC_{t,T,i}^{\prime}, h_T^E\right):=\varphi \operatorname{Tr}_s\left(\frac{N^Z}{2} f^{\prime}\left(\ttD_{t,T,i}\right)\right)-\frac{\chi^{\prime}(Z_i, F_i)}{2}.
	\end{align*}
	Set
	\begin{align*}
		&\ \ \ \ \mathcal{T}_{\la,i}^{\mL}\left(T^H M_i, g^{T Z_i}, h^{F_i}\right)&\\
		&=-\int_1^{\infty}\left(f_{\la}^{\wedge}\left(C_{t,i}^{\prime}, h^{E_i}\right)- \frac{  \chi(Z_i,F_i)\dim(Z)-2\chi'(Z_i,F_i)}{4} f^{\prime}\left(\frac{\sqrt{-1} \sqrt{t}}{2}\right)\right)\frac{d t}{t},
	\end{align*}
	\begin{align*}
		&\ \ \ \ \mathcal{T}_{\la,i}^{\mS}\left(T^H M_i, g^{T Z_i}, h^{F_i}\right)\\
		&=-\int_0^{1}\left(f_{\la}^{\wedge}\left(C_{t,i}^{\prime}, h^{{E_i}}\right)
		-\frac{  \chi(Z_i,F_i)\dim(Z)-2\chi'(Z_i,F_i)}{4} f^{\prime}\left(\frac{\sqrt{-1} \sqrt{t}}{2}\right)\right) \frac{d t}{t}.
	\end{align*}
	
	Similarly, one has $\mathcal{T}_{\la,i}^{\mL}\left(T^H \bm_i, g^{T \bz_i}, h_T^{\F_i}\right)(T)$ and $\mathcal{T}_{\la,i}^{\mS}\left(T^H \bm_i, g^{T \bz_i}, h_T^{\F_i}\right)(T)$.
	
	One can see that for ``$\bullet$"=``$\mL$'' or ``$\mS$'',
	$$ \mathcal{T}_{\la,i}^{\bullet}\left(T^H M_i, g^{T Z_i}, h^{F_i}\right)= \mathcal{T}_{i}^{\bullet}\left(T^H M_i, g^{T Z_i}, h^{F_i}\right),$$ 
	$$\mathcal{T}_{\la,i}^{\bullet}\left(T^H \bm_i, g^{T \bz_i}, h_T^{\F_i}\right)(T)=\mathcal{T}_{i}^{\bullet}\left(T^H \bm_i, g^{T \bz_i}, h_T^{\F_i}\right)(T).$$

	\begin{thm}\label{limitfhat}\label{thmfhat}
		$$\lim_{T\to\infty}f_{\la}^{\wedge}\left(C_{t,T}^{\prime}, h^E\right)=\lim_{T\to\infty}\sum_{i=1}^2f_{\la}^{\wedge}\left(\tC_{t,T,i}^{\prime}, h_T^\bE\right)=\sum_{i=1}^2f_{\la}^{\wedge}\left(C_{t,i}^{\prime}, h^E\right).$$ More precisely, for example, for every $\epsilon>0$, there exists $T_0(t,\epsilon)>0$, such that whenever $T\geq T_0(t,\epsilon),$
		\[\|f_{\la}^{\wedge}\left(C_{t,T}^{\prime}, h^E\right)-\sum_{i=1}^2f_{\la}^{\wedge}\left(C_{t,i}^{\prime}, h^E\right)\|_{L^\infty}<\epsilon.\]
		Here we put a metric $g^{TS}$ on $S$, and for $\a\in\Omega(S),$
		\[\|\a\|_{L^\infty}:=\sup_{\s\in S}|\a|(\s),\]
		where $|\a|:=\sqrt{g^{TS}(\a,\a)}.$
	\end{thm}

	\begin{thm}\label{larcon}
		\begin{align*}&\ \ \ \ \lim_{T\to\infty} \mathcal{T}_{\la}^{\mL}\left(T^H \bm, g^{T \bz}, h^{\F}\right)(T)=\lim_{T\to\infty}\sum_{i=1}^2\mathcal{T}_{\la,i}^{\mL}\left(T^H \bm_i, g^{T \bz_i}, h_T^{\F_i}\right)(T)\\
			&=\sum_{i=1}^2\mathcal{T}_{\la,i}^{\mL}\left(T^H M_i, g^{T Z_i}, h^{F_i}\right).\end{align*}
		The limit is taken w.r.t. the topology described in Theorem \ref{limitfhat}.
	\end{thm}

	\begin{thm}\label{main}
		\[\mathcal{T}_{\la}^{\mS}\left(T^H \bm, g^{T \bz}, h^{\F}\right)(T)=\sum_{i=1}^2\mathcal{T}_{\la,i}^{\mS}\left(T^H M_i, g^{T Z_i}, h^{F_i}\right)-(T-\log(2))\chi(Y)\rank(F)/2+o(1),\]
		\[\mathcal{T}_{\la,i}^{\mS}\left(T^H \bm_i, g^{T \bz_i}, h_T^{\F_i}\right)(T)=\mathcal{T}_{\la,i}^{\mS}\left(T^H M_i, g^{T Z_i}, h^{F_i}\right)-T\log(2)\chi(Y)\rank(F)/4+o(1).\]
		as $T\to\infty.$
		
		As a result,
		\[\mathcal{T}_{\la}^{\mS}\left(T^H \bm, g^{T \bz}, h^{\F}\right)(T)=\sum_{i=1}^2\mathcal{T}_{\la,i}^{\mS}\left(T^H \bm_i, g^{T \bz_i}, h^{\F_i}_T\right)(T)+\log(2)\chi(Y)\rank(F)/2+o(1).\]
		The limit is taken w.r.t. the topology described in Theorem \ref{limitfhat}.
	\end{thm}

	Next, we have the following Mayer-Vietoris exact sequence (c.f. \cite[(0.16)]{bruning2013gluing}) of vector bundles over $S$:
	\be\label{mv}
	\mathcal{MV}:\cdots \stackrel{\p_{k-1}}\rightarrow H_{\text {rel }}^k\left(\bz_2;\F\right)  \stackrel{e_k}{\rightarrow}H^k\left(\bz;\F\right)  \stackrel{r_k}{\rightarrow} H_{\mathrm{abs}}^k\left(\bz_1;\F\right)  \stackrel{\p_k}{\rightarrow} \cdots.
	\ee
	\def\tO{{\tilde{\Omega}}}
	\def\tp{{\tilde{\p}}}
	\def\tP{{\tilde{\mathcal{P}}}}
	\def\tir{{\tilde{r}}}
	\def\te{{\tilde{e}}}
	\def\mv{\mathcal{MV}}
	Let $\tO_{\sm}(\bz,\F)(T)$ be the space generated by eigenforms of $\tD_T$ for eigenvalues inside $[0,\delta]$, then by our discussion above in \cref{witwei}, $$\tO_{\sm}(\bz,\F)(T)=e^{p_T}\Omega_{\sm}(\bz,\F)(T).$$ And $\tP^\delta(T)=e^{p_T}\P^{\delta}(T)e^{-p_T}$ is the orthogonal projection w.r.t. $ \tO_{\sm}(\bz,\bar{F})(T)$.
	Let $\H(\bz;\F)(T):=\ker(\tD_{T})$, and $\H(\bz_i;\F_i)(T):=\ker(\tD_{T,i})$. 
	We also have the following Mayer-Vietoris exact sequence induced by Hodge theory and (\ref{mv})
	\be\label{mvp}
	\mathcal{MV}(T):\cdots \stackrel{\p_{k-1,T}}\rightarrow \H^k\left(\bz_2;\F_2\right)(T)  \stackrel{e_{k,T}}{\rightarrow}\H^k\left(\bz;\F\right)(T)  \stackrel{r_{k,T}}{\rightarrow} \H^k\left(\bz_1;\F_1\right)(T)  \stackrel{\p_{k,T}}{\rightarrow} \cdots 
	\ee
	with metric and flat connections induced by Hodge theory. Let $\mathcal{T}(T)$ be the analytic torsion form for this complex.
	
	For any $L^2$ differential form $w$ on $Z_i$ (or $\bz_i$), let $\cE(w)$ be the extension of $w$, s.t. outside $Z_i$ (or $\bz_i$), $\cE(w)=0.$ We will not distinguish $w$ and $\cE(w)$ if the context is clear.
	
	\subsection{A short exact sequence induced by small eigenvalues}
	Next, when $T$ is large enough, we have a sequence of morphism of vector bundles
	\be\label{short}0\to \H^k(\bz_2,\F_2)(T)\stackrel{\te_{k,T}}\rightarrow \tO_{\sm}^k(\bz,\F)(T)\stackrel{\tir_{k,T}}\rightarrow\H^k(\bz_1,\F_1)(T)\to0.\ee
	
	Here 
	$$\te_{k,T}(u):=\tP^{\delta}(T)\cE(u),\forall u\in \H^k(\bz_2,\F_2)(T);$$ 
	$$\tir_{k,T}(u):=\tP_{1}(T)(u|_{\bz_1}),\forall u\in\tO_{\sm}^k(\bz,\F)(T).$$
	Recall that $\tP_{i}(T)$ is the orthogonal projection  w.r.t. $\ker(\tilde\Delta_{T,i})$ ($i=1,2$). 
	

	Let $\eta \in C_c^{\infty}([0,1])$, such that $\eta|_{[0,1 / 4]} \equiv 0,\left.\eta\right|_{[1 / 2,1]} \equiv 1$.  Let $Q_T: \Omega_{\mathrm{abs}}\left(Z_1 ; F_1\right) \oplus \Omega_{\mathrm{rel}}\left(Z_2 ; F_2\right) \rightarrow W^{1,2}\Omega(\bar{Z}, \bar{F})$ be
	$$
	Q_T(u)(x):=\left\{\begin{array}{l}
		u_i(x), \text { if } x \in Z_i ,\\
		\eta(-s) u_1(-1, y) e^{-p_T(s)-T / 2}, \text { if } x=(s, y) \in[-1,0] \times Y, \\
		\eta(s) u_2(1, y) e^{p_T(s)-T / 2}, \text { if } x=(s, y) \in[0,1] \times Y ,
	\end{array}\right.
	$$
	for $u=\left(u_1, u_2\right)\in \Omega_{\mathrm{abs}}\left(Z_1 ; F_1\right) \oplus \Omega_{\mathrm{rel}}\left(Z_2 ; F_2\right)$.
	Suppose $u_i=\a_i+\b_i\wedge ds(i=1,2)$ on the tube $[-2,2]\times Y$, moreover, $D^{\bz}_T$ is decomposed as $D_T^\R+D^Y$, where $D_T^\R=ds\wedge\nabla_{\frac{\p}{\p s}}-i_{\frac{\p}{\p s}}\nabla_{\frac{\p}{\p s}}.$
	It follows from the construction of $Q_T$ and the boundary conditions that ,\begin{equation}\label{qtdt}D^{\bz}_TQ_T(u)(x)=\begin{cases}D^{Z_i}u_i, \mbox{ if }x\in Z_i;\\
			D^Y\a_1(-1,y)e^{-f_T-T/2},\mbox{ if } x=(s,y)\in[-1,-1/2]\times Y; \\
			D^Y\b_2(1,y)\wedge ds e^{f_T-T/2}, \mbox{ if }x=(s,y)\in[1/2,1]\times Y.
	\end{cases}\end{equation}
	
	Set $\Q_T:=e^{p_T}Q_T.$ 
	One can see that
	\begin{prop}\label{last1p} For $u \in \operatorname{ker}\left(\Delta_1\right) \oplus \operatorname{ker}\left(\Delta_2\right)$,
		$$
		\begin{array}{c}
			\left\|Q_T u-\cE(u)\right\|_{L^2}^2 \leq\frac{C}{\sqrt{T}}\|u\|^2_{L^2}, \\
			\left\|\mathcal{P}^\delta(T) Q_T(u)-\cE(u)\right\|_{L^2}^2 \leq \frac{C\|u\|_{L^2}^2}{\sqrt{T}}.
		\end{array}
		$$
		for some constant $C$ that is independent of $T$. 
		
		As a result,
		$$
		\begin{array}{c}
			\left\|\Q_T u-\cE(e^{p_T}u)\right\|_{L^2,T}^2 \leq\frac{C}{\sqrt{T}}\|e^{p_T}u\|^2_{L^2,T}, \\
			\left\|\tP^\delta(T) \Q_T(u)-\cE(e^{p_T}u)\right\|_{L^2,T}^2 \leq \frac{C\|e^{p_T}u\|_{L^2,T}^2}{\sqrt{T}}.
		\end{array}
		$$
		Recall that $\|\cdot\|_{L^2,T}$ is the norm induced by $g^{T\bz}$ and $h^{\F}_T:=e^{-2p_T}h^{\F}.$
		
		As a result, when $T$ is large enough, $\tP^\delta(T) \Q_T(u)$ spans $\tO_{s m}(\bar{Z}, \bar{F})(T)$ for $u \in \operatorname{ker}\left(\Delta_1\right) \oplus \operatorname{ker}\left(\Delta_2\right)$.
	\end{prop}
	\begin{proof} For $u=(u_1,u_2) \in \operatorname{ker}\left(\Delta_1\right) \oplus \operatorname{ker}\left(\Delta_2\right)$, set $u_T=\mathcal{P}^\delta(T) Q_T(u), v_T=Q_T(u)-u_T$.
		First, by \cite[Lemma 4.2]{Yantorsions} and the compactness of $S$,
		\begin{align}\begin{split}\label{6}
				&\ \ \ \ \int_Y\left|u_i\left((-1)^i, y\right)\right|^2+\left|D^{\bz}_T u_i\left((-1)^i, y\right)\right|^2 \operatorname{dvol}_Y \leq C \int_{Z_i}\left|u_i\right|^2 \mathrm{dvol}.
		\end{split}\end{align}
		for some constants $C, C^{\prime}$ that doesn't depends on $T$ and $\s$.
		By (\ref{6}) and a straightforward computation, one can see that
		\be\label{000}
		\left\|Q_T u-\cE(u)\right\|_{L^2}^2 \leq\frac{C}{\sqrt{T}}\|u\|^2_{L^2}
		\ee
		and
		\be\label{8}
		\left\|D^{\bz}_T Q_T u\right\|_{L^2}^2 \leq \frac{C}{\sqrt{T}}\|u\|^2_{L^2} .
		\ee
		
		Moreover,
		\be\label{9}
		\delta\left\|v_T\right\|_{L^2}^2 \leq\left\|D^{\bz}_T v_T\right\|_{L^2}^2 \leq\left\|D^{\bz}_T Q_T u\right\|_{L^2}^2 .
		\ee
		(\ref{8}) and (\ref{9}) then imply that
		$$
		\left\|v_T\right\|_{L^2}^2 \leq \frac{C}{\delta \sqrt{T}}\|u\|_{L^2},
		$$
		i.e.,
		\be\label{001}
		\left\|\mathcal{P}^\delta(T) Q_T(u)-Q_T(u)\right\|_{L^2}^2 \leq \frac{C\|u\|_{L^2}^2}{\delta \sqrt{T}} .
		\ee
		The proposition then follows form (\ref{000}) and (\ref{001}).
	\end{proof}
	
	Notice that if $u\in\Omega_{\bd}(Z_i,F_i)$, then $\Q_Tu\in\Omega_{\bd}(\bz_i,\F_i)$. By Hodge theory and Theorem \ref{eigencon}, when $T$ is big enough, all eigenvalues of $\tD_{T,i}$ inside $[0,\delta]$ must be $0.$ Let $\tP_i(T)$ be the orthogonal projection w.r.t. $\ker(\tD_{T,i})$. Similarly, one has
	\begin{prop} \label{last2p}For $u \in \operatorname{ker}\left(\Delta_i\right)$,
		$$
		\begin{array}{c}
			\left\|\Q_T u-\cE(e^{p_T}u)\right\|_{L^2,T}^2 \leq\frac{C}{\sqrt{T}}\|e^{p_T}u\|^2_{L^2,T}, \\
			\left\|\tP_{i}(T) \Q_T(u)-\cE(e^{p_T}u)\right\|_{L^2,T}^2 \leq \frac{C\|e^{p_T}u\|_{L^2,T}^2}{\sqrt{T}}.
		\end{array}
		$$
		As a result, when $T$ is large enough, $\tP_i(T) \Q_T(u)$ spans $\H(\bar{Z}_i, \bar{F}_i)(T)$ for $u \in \operatorname{ker}\left(\Delta_i\right)$.
	\end{prop}
	
	\def\bE{{\bar{E}}}

	\begin{prop}\label{last4p1}
		$\te_{k,T}$ and $\tir^*_{k,T}$ are almost isometric embeddings as $T\to\infty$, where $\tir^*_{k,T}$ is the adjoint of $\tir_{k,T}$. That is, for example, for any $u\in \H^k(\bz_2;\F_2)(T)$, $\lim_{T\to\infty}\frac{\|\te_{k,T}u\|_{L^2,T}}{\|u\|_{L^2,T}}=1.$ 
	\end{prop}
	\begin{proof}\ \\
		$\bullet$ \textit{ $\te_{k,T}$ is almost isometric.}
		
		For any $u\in \H^k(\bz_2,\F_2)(T)$, there exists $u_T\in \ker(\Delta_2)\cap\Omega^k_{\rel}(Z_2,F_2)$ such that
		$u=\tP_i(T) \Q_T(u_T)$, then by Proposition \ref{last2p}, \be\label{lasteq1}\|u\|^2_{L^2,T}\geq (1-\frac{C}{\sqrt{T}})\|e^{p_T}u_T\|^2_{L^2,T}.\ee
		
		While by Proposition \ref{last2p}, Proposition \ref{last1p} and the fact that $\|\tP^\delta(T)\|=1$,
		\begin{align}\label{lasteq2}
			\begin{split}
				&\ \ \ \ \|\te_{k,T}u\|^2_{L^2,T}=\|\tP^\delta(T)\cE u\|^2_{L^2,T}\\
				&\leq \|\tP^\delta(T) \Q_Tu_T\|^2_{L^2,T}+\frac{C\|e^{p_T}u_T\|^2_{L^2,T}}{\sqrt{T}}\\
				&\leq \|e^{p_T}u_T\|^2_{L^2,T}(1+\frac{C'}{\sqrt{T}})
			\end{split}
		\end{align}
		It follows from (\ref{lasteq1}) and (\ref{lasteq2}) that
		\[\limsup_{T\to\infty}\frac{\|\te_{k,T}u\|_{L^2,T}^2}{\|u\|_{L^2,T}}=1.\]
		Similarly, \[\liminf_{T\to\infty}\frac{\|\te_{k,T}u\|_{L^2,T}^2}{\|u\|_{L^2,T}}=1.\]
		$\bullet$ \textit{ $\tir^*_{k,T}$ is almost isometric.}\\
		For $u\in \H^k(\bz_1,\F_1)(T),$ we first show that \be\label{rstar}\tir^*_{k,T}u=\tP^{\delta}(T)\cE(u).\ee
		Just notice that for any $v\in \tO_{\sm}(\bz,\F)(T)$, 
		\begin{align*}
			(\tir_{k,T}v,u)_{L^2(\bz_2),T}=(v,\cE(u))_{L^2(\bz),T}=(v,\tP^{\delta}(T)\cE(u))_{L^2(\bz),T}.
		\end{align*}
		Following the same steps as above, one can show that $\tir^*_{k,T}$ is almost isometric. 
		

		
	\end{proof}
	\begin{thm}\label{exathm}
		With maps $\te_{k,T}$ and $\tir_{k,T}$ given above, the sequence (\ref{short}) is exact.
	\end{thm}
	\begin{proof}
		Let $\lan\cdot,\cdot\ran_T$ be the pointwise inner product on each fiber that is induced by $g^{T\bz}$ and $h_T^{\F}.$\\ \ \\
		$\bullet$ In follows from Proposition \ref{last4p1} that $\te_{k,T}$ and $\tir^*_{k,T}$ are injective when $T$ is large.\\ \ \\
		$\bullet$ \textit{$\cE(\ker(\tilde{\Delta}_{T,1}))\subset \ker(\delta^{\bz,*}_T)$ and $\cE(\ker(\tilde{\Delta}_{T,2}))\subset\ker(d^{\bz})$:}\\ \ \\
		Let $u\in\ker(\tD_{T,2})$. First, since $u$ satisfies relative boundary conditions, integration by parts shows that for any $\beta\in \Omega(\bz,\F)$, $\int_{\bz}\lan \cE(u),\delta^{\bz,*}_T\beta\ran_T\dvol=0$. Thus, $\cE(u)\in\ker(d^{\bz})$.
		Similarly, $\cE(\ker(\tilde{\Delta}_{T,1}))\subset \ker(\delta^{\bz,*}_T)$.
		\\ \ \\
		$\bullet$ \textit{$\im\  \te_{k,T}=\ker \tir_{k,T}$:}\\
		For the dimension reason, it suffices to show that $\im \te_{k,T}\subset\ker \tir_{k,T}$. That is, it suffices to show $\im \te_{k,T}\perp \im \tir^*_{k,T}$. First, for any $u_i\in\ker(\tilde\Delta_{T,i})$, $i=1,2$, it's clear that \be\label{eq200}(\cE(u_1),\cE(u_2))_{L^2,T}=0.\ee
		
		Since $\cE(u_1)\in \ker(\delta^{\bz,*}_T),\cE(u_2)\in\ker(d^{\bz})$, one can see that
		$(1-\tP^{\delta}(T))u_1\in\im\ \delta^{\bz,*}_T, (1-\tP^{\delta}(T))u_1\in\im\ d^{\bz}$, which means
		\be\label{eq201}((1-\tP^{\delta}(T))\cE(u_1),(1-\tP^{\delta}(T))\cE(u_2))_{L^2,T}=0.\ee
		By (\ref{rstar}), the definition of $\te_{k,T}$, (\ref{eq200}) and (\ref{eq201}),
		\[(\te_{k,T}u_2,\tir^*_{k,T}u_1)_{L^2,T}=0.\]
	\end{proof}
	
	Moreover, we have the following complexes of finite dimensional vector bundles
	\begin{align}\label{com1}&0\to \H^0(\bz_i,\F_i)(T)\stackrel{0}\rightarrow\H^1(\bz_i,\F_i)(T)\stackrel{0}\rightarrow  \cdots\stackrel{0}\rightarrow\H^{\dim Z}(\bz_i,\F_i)(T)\to 0 \\
		&\label{com2}0\to \tO_{\sm}^0(\bz,\F)(T)\stackrel{d^{\bz}}\rightarrow\tO_{\sm}^1(\bz,\F)(T)\stackrel{d^{\bz}}\rightarrow  \cdots\stackrel{d^{\bz}}\rightarrow\tO_{\sm}^{\dim Z}(\bz,\F)(T)\to 0.
	\end{align}
	
	Integration by parts as in the proof of Theorem \ref{exathm}, one can show easily that
	\begin{prop}\label{exapp}
		$d^{\bz}\circ\te_{k,T}=0,\tir_{k,T}\circ d^{\bz}=0.$
	\end{prop}
	
	Hence, by Theorem \ref{exathm} and Proposition \ref{exapp}, we get the following long exact sequence again
	\be\label{mvpp}
	\cdots \stackrel{\p_{k-1,T}}\rightarrow \H^k\left(\bz_2;\F_2\right)(T)  \stackrel{e_{k,T}}{\rightarrow}\H^k\left(\bz;\F\right)(T)  \stackrel{r_{k,T}}{\rightarrow} \H^k\left(\bz_1;\F_1\right)(T)  \stackrel{\p_{k,T}}{\rightarrow} \cdots 
	\ee
	with metric and connection induced by Hodge theory.

	Let \begin{align*}&\ \ \ \ \T_{\sm}(T^H\bm,g^{T\bz},h^{\F})(T)\\
		&:=-\int_0^\infty \left(f_{\sm}^{\wedge}(C_{t,T}',h^{\bE})-\frac{\chi'(Z,F)}{2}\right)+\frac{\chi'(Z,F)-d(\Omega_{\sm}(\bm,\F)(T))}{2}f'(\frac{i\sqrt{t}}{2})\frac{dt}{t}.\end{align*}

	
	The following Theorem will be proved in \cref{lastreal}.
	
	\begin{thm}\label{int6p}$\lim_{T\to\infty}\T_{\sm}(T^H\bm,g^{T\bz},h^{\F})(T)-\T(T)=0.$\end{thm}
	

	It follows from anomaly formulas (c.f. \cite[Theorem 2.24 and Theorem 3.24]{bismut1995flat} and \cite[Theorem 1.5]{JL})  that
	
	\begin{thm}\label{int1p}
		In $Q^{\mS} / Q^{S}_0$,
		\begin{align*}
			&\ \ \ \ \T(T^H\bm, g^{T\bz},h^{\F})(T)-\sum_{i=1}^2\T_i(T^H\bm_i,g^{T\bz_i},h_T^{\F_i})(T)-\T(T)\\
			&=\T(T^HM,g^{TZ},h^{F})-\sum_{i=1}^2\T_i(T^HM_i,g^{TZ_i},h^{F_i})-\T.
		\end{align*}
	\end{thm}

	\begin{proof}[Proof of Theorem \ref{main0}]
		It follows from Theorem \ref{larcon}, \ref{main} and \ref{int6p} that
		\begin{align*}
			&\ \ \ \ \T(T^H\bm,g^{T\bz},h^{\F})(T)-\sum_{i=1}^2\T_i(T^H\bm_i,g^{T\bz_i},h_T^{\F_i})(T)-\T(T)
			\\&=\log(2)\chi(Y)\rank(F)/2+o(1).
		\end{align*}
		
		Thus, by Theorem \ref{int1p}, Theorem \ref{main0} follows.
	\end{proof}

	\section{Convergence of Eigenvalues}\label{eigen}
	First, it's straightforward to check that
	\begin{lem}\label{equicon}
		Let $F\to X$ be a flat complex vector bundle over a compact smooth manifold $X$, $f$ be a smooth function on $X$. Let $h_l^{F}, g^{TX}_l, \nabla^F_l$ be smooth families of metrics and connections over $F\to X$, $l\in [0,1]$. Let $d_l:\Omega^*(X;F)\to \Omega^{*+1}(X;F)$ be the covariant derivative w.r.t. $\nabla_l^F$, and $d_l^{*} $ be the adjoint of $d_l^F$. Let $d_{f,l}:=d_l^F+df\wedge$, and $d_{f,l}^*$ be the adjoint of 
		$d_{f,l}$. Then for any $u\in\Omega,\epsilon>0$, there exists $\delta>0$ that doesn't depend on $u$ and $f$, such that whenever $|l_1-l_2|<\delta$, one has
		\[\frac{\int_X|d_{f,l_1}u|_{l_1}^2+|d_{f,l_1}^*u|_{l_1}^2}{\int_X|u|_{l_1}^2}\leq  (1+\epsilon)\frac{\int_X|d_{f,l_2}u|_{l_2}^2+|d_{f,l_2}^*u|_{l_2}^2}{\int_X|u|_{l_2}^2}.\]
		Here $|\cdot|_{l}$ is the metric on $\Lambda^*(TZ)\otimes F$ induced by $h^F_l$ and $g^{TX}_l.$
	\end{lem}


	First, one observes that $\l_k(T,\s)$ has uniform upper bounds:
	\begin{lem}\label{o1}
		Fix $k\in\mathbb{Z}^+.$
		There exists an increasing sequence $\{\Lambda_k\}_{k=1}^\infty$ of constants, such that $\lambda_{k}(T,\s)\leq \Lambda_k$.
	\end{lem}
	\begin{proof}
		For a fixed $\s\in S$, it follows from \cite[Lemma 4.1]{Yantorsions} that there exists $\{\Lambda_k(\s)\}$, such that $\l_{k}(T,\s)\leq \Lambda_{k}(\s).$ It follows from Lemma \ref{equicon} that when $\s'$ is close to $\s$, $\l_{k}(T,\theta')\leq 2\l_{T,\theta}.$ The existence of $\Lambda_k$ then follows from the compactness of $S$.
	\end{proof}
	\begin{cor}\label{equicon1}
		$\l_k(T,\theta)$ is a family of equicontinuous function on $S$. That is, for any $\s\in S,\epsilon>0$, there exists a neighborhood $U$ of $\s$ that doesn't depends on $T$, whenever $\s'\in S$, $|\l_k(T,\s)-\l_k(T,\s')|<\epsilon.$
	\end{cor}
	\begin{proof}
		By Lemma \ref{equicon} and Lemma \ref{o1}, when $\s'$ is closed to $\s$, $|\l_k(T,\s)-\l_{k}(T,\s')|\leq \epsilon \l_k(T,\s)\leq  \epsilon \Lambda_k.$
	\end{proof}

	Next, it follows from \cite[Lemma 4.2]{Yantorsions} and compactness of $S$ that,
	\begin{lem}\label{limit1}
		
		Let $u\in \Omega(\bz;\F)$ be a unit eigenform w.r.t. an eigenvalue $\leq \l$. Then for $s\in[-2,2]$
		\[\int_{Y}|u|^2(s,y)+|D^{\bz}_Tu|^2(s,y)\dvol_Y\leq C(1+\lambda^2)\]
		if $T$ is large enough. Here the constant $C$ is independent of $T$ and $\s,$ $|\cdot|$ is the metric on $\Lambda^*(\bz)\otimes \F|_{\bz}$ induced by $h^{\F}$ and $g^{T\bz}$.

	\end{lem}

	\begin{proof}[Proof of Theorem \ref{eigencon}]
		Fix $\theta\in S$, \cite[Theorem 3.1]{Yantorsions} implies that $\lim_{T\to\infty}\l_k(T,\s)=\l_k(\s)$. Since $S$ is compact, the uniformness follows from Corollary \ref{equicon1} and continuity of $\l_k(\s)$. 
		
		Similarly, one can show $\lim_{T\to\infty}\tl_k(T,\s)=\l_k(\s)$.
	\end{proof}
	
	Recall that for $u=(u_1,u_2)\in \Omega_{\abs}(Z_1,F_1)\oplus\Omega_{\rel}(Z_2, F_2)$,
	\begin{equation*}
		Q_T(u)(x):=\begin{cases}
			u_i(x), \mbox{ if  $x\in Z_i$,}\\
			\eta(-s)u_1(-1,y)e^{-p_T(s)-T/2}, \mbox{if $x=(s,y)\in [-1,0]\times Y$,}\\
			\eta(s)u_2(1,y)e^{p_T(s)-T/2}, \mbox{if $x=(s,y)\in [0,1]\times Y$.}
		\end{cases}
	\end{equation*}
	It follows from Lemma \ref{limit1} and the construction of $Q_T$ that
	\begin{lem}\label{qtuu}Let $u_i\in \Omega(Z_i;F_i)$ be a unit eigenform w.r.t. an eigenvalue $\leq \l$(i=1,2). Then
		\[\|Q_T(u)-\cE(u)\|^2_{L^2}\leq \frac{C(1+\l^2)}{\sqrt{T}}\|u\|^2_{L^2}.\]
		Here the constant $C$ is independent of $T$ and $\s.$
	\end{lem}
	
	It follows from \cite[Lemma 5.1 and Lemma 5.2]{Yantorsions} and the compactness of $S$ that
	\begin{lem}\label{est1}
		There exists constants $c_1,c_2,c_3,c_4$ and $c_5$ independent of $T$ and $\s$, such that $\l_k(T,\s)\geq u_k(T).$ Here $\{u_k(T)\}_{k=1}^\infty$ is the collection of $4$ copies of $\{v_l(T)+c_4m^{2/(\dim Z-1)}\}_{l=1,m=1}^\infty$ and $2$ copy of $\{c_5 l^{2/\dim Z}\}$, listed in the increasing order and counted with multiplicity. Moreover,  $\{v_k(T)\}_{k=1}^\infty$ is the collection of $\{T\max\{c_1l-c_2,0\}\}_{l=1}^\infty$ and $\{c_3l^2\}_{l=1}^\infty$, listed in the increasing order and counted with multiplicity.
	\end{lem}
	
	\subsection{Convergence of $f_{\la}^{\wedge}\left(C_{t,T}^{\prime}, h^{\bE}\right)$ and the large time contributions}
	Let $V_T=d^{\bz}_T-d^{\bz,*}_T$ and $\F_{t}:=D_{t,T}-\frac{\sqrt{t}}{2}(d^{\bz}_T-d^{\bz,*}_T)$, then all eigenvalues of $V_T$ are pure imaginary. Moreover, $\F_t$ is nilpotent. 
	Similarly, one has $F_{t,i}$ and $F_{t}.$ Moreover, \be\label{limithat1}\F_t|_{Z_i}=F_{t,i}.\ee
	
	We define $\F^j_t$ inductively: set $\F^0_t=\F_t$, and $\F^{j+1}_t=[V_T,\F^j_t].$
	Since $T^H\bm,g^{T\bz}$ and $h^{\F}$ are product-type, by a straightforward computation,
	\begin{lem}\label{fhat0}
		There exists $(T,\s)$-independent $C_j>0$, s.t. $\|\F^j_t\|\leq C_j(1+\sqrt{t}^{-1}).$
	\end{lem}
	
	It's trivial that
	\begin{lem}\label{tri}
		If $\tau$ is pure imagenary and $|\Re(\l)|=1$, then 
		\[|\l-\tau|^{-1}\leq C\frac{|\l|}{|\tau|}\]
		or 
		\[|\l-\tau|^{-1}\leq 1\]
		for some universal constant $C.$
		
	\end{lem}
	
	\begin{lem}\label{fhat1}
		Let $u$ be a unit eigenform w.r.t. an eigenvalue $\mu$ for $V_T$ (Moreover, assume that $|\mu|>0$), then for any $j\in\Z^+,t>0$, there exits $(T,\s)$-independent $C_j>0$, such that
		\[\left|\left(\left(f'(D_{t,T})-f'(\sqrt{t}V_T)\right)u,u\right)_{L^2}\right|\leq\frac{C_j(1+\sqrt{t}^{-\dim(S)})}{\sqrt{t}|\mu|^{j}}.\]
		Similarly, one can show that there exits $(T,\s)$-independent $C'_j>0$, such that
		\[\left|\left((-1)^{N^Z}N^Z\left(f'(D_{t,T})-f'(\sqrt{t}V_T)\right)u,u\right)_{L^2}\right|\leq\frac{C'_j(1+\sqrt{t}^{-\dim(S)})}{\sqrt{t}|\mu|^{j}}.\]
	\end{lem}
	\begin{proof}
		Let $\gamma$ be the oriented contour given by $\{z\in\C:|\Re(z)|=1\}.$
		Proceeding as in the proof of \cite[Theorem 2.13]{bismut1995flat},
		\begin{align}\begin{split}\label{fhat11}
				&\ \ \ \ f^{\prime}\left(D_{t,T}\right)-f^{\prime}\left(\sqrt{t}V_T\right)\\
				&=\sum_{l=1}^{\dim S}\int_{\gamma}f'(\l)\left((\l-\sqrt{t}V_T)^{-1}\F_t\right)^l(\l-\sqrt{t}V_T)^{-1}d\l.
		\end{split}\end{align}
		
		First, notice that
		\begin{align}\begin{split}\label{fhat111}
				&\ \ \ \ \left|\left(\int_{\gamma}f'(\l)(\l-\sqrt{t}V_T)^{-1}\F_t(\l-\sqrt{t}V_T)^{-1}ud\l,u\right)_{L^2}\right|\\
				&=\left|\left(\int_{\gamma}f'(\l)\F_t(\l-\sqrt{t}V_T)^{-1}ud\l,(\bar{\l}+\sqrt{t}V_T)^{-1}u\right)_{L^2}\right|\\
				&=\left|\left(\int_{\gamma}f'(\l)(\l-\sqrt{t}\mu)^{-2}\F_tud\l,u\right)_{L^2}\right|=0.
		\end{split}\end{align}
		Recall that $\F^0_t:=\F_t$, and $\F^{j+1}_t:=[V_T,\F^j_t].$
		Through a simple calculation,
		\be\label{fhat12}
		[\F_t^j,(\l-\sqrt{t}V_T)^{-1}]=\pm(\l-\sqrt{t}V_T)^{-1}\sqrt{t}\F_t^{j+1}(\l-\sqrt{t}V_T)^{-1}.
		\ee
		Consequently, by (\ref{fhat12}), Lemma \ref{tri} and Lemma \ref{fhat0}, if $\l\in\gamma$,
		\begin{align}\begin{split}\label{fhat13}
				&\ \ \ \ \left(\left((\l-\sqrt{t}V_T)^{-1}\F_t\right)^2(\l-\sqrt{t}V_T)^{-1}u,u\right)\\
				&=\left(\sum_{k=1}^{j-1}(\l-\sqrt{t}V_T)^{-(k+1)}\sqrt{t}^{k}\F_t^{k-1}\F_t(\l-\sqrt{t}V_T)^{-1}\right.\\
				&\left.+(\l-\sqrt{t}V_T)^{-j}\sqrt{t}^{j}\F_t^{j-1}((\l-\sqrt{t}V_T)^{-1})\F_t(\l-\sqrt{t}V_T)^{-1}u,u\right)\\
				&\leq \left(\sum_{k=1}^{j-1}(\l-\sqrt{t}V_T)^{-(k+1)}\sqrt{t}^{k}\F_t^{k-1}\F_t(\l-\sqrt{t}V_T)^{-1}u,u\right)\\
				&+C(1+\sqrt{t}^{-2})\left|(\l-\sqrt{t}\mu)^{-(j+1)}\sqrt{t}^{j}\right|\\
				&\leq \left(\sum_{k=1}^{j-1}(\l-\sqrt{t}V_T)^{-(k+1)}\sqrt{t}^{k}\F_t^{k-1}\F_t(\l-\sqrt{t}V_T)^{-1}u,u\right)\\
				&+C(1+\sqrt{t}^{-2})|\l|^{j+1}|\mu|^{-(j+1)}\sqrt{t}^{-1};\\
		\end{split}\end{align}
		similarly,
		\begin{align}\begin{split}\label{fhat131}
				&\ \ \ \ \left(\left((\l-\sqrt{t}V_T)^{-1}\F_t\right)^2(\l-\sqrt{t}V_T)^{-1}u,u\right)\\
				&\geq \left(\sum_{k=1}^{j-1}(\l-\sqrt{t}V_T)^{-(k+1)}\sqrt{t}^{k}\F_t^{k-1}\F_t(\l-\sqrt{t}V_T)^{-1}u,u\right)\\
				&-C(1+\sqrt{t}^{-2})|\l|^{j+1}|\mu|^{-(j+1)}\sqrt{t}^{-1}.\\
		\end{split}\end{align}

		By (\ref{fhat13}) and (\ref{fhat131}), proceeding as in (\ref{fhat111}), one can see that
		
		\begin{align}\begin{split}
				&\ \ \ \ \left|\left(\int_{\gamma}f'(\l)\left((\l-\sqrt{t}V_T)^{-1}\F_t\right)^2(\l-\sqrt{t}V_T)^{-1}ud\l,u\right)_{L^2}\right|\leq \frac{C_j(1+\sqrt{t}^{-2})}{\sqrt{t}|\mu|^{j+1}}.
		\end{split}\end{align}
		
		Similarly, one can show
		\begin{align}\begin{split}
				&\ \ \ \ \left|\left(\int_{\gamma}f'(\l)\left((\l-\sqrt{t}V_T)^{-1}\F_t\right)^l(\l-\sqrt{t}V_T)^{-1}ud\l,u\right)_{L^2}\right|\leq \frac{C_{j,l}(1+\sqrt{t}^{-l})}{\sqrt{t}|\mu|^{j+1}}.
		\end{split}\end{align}
	\end{proof}
	We also have
	\begin{lem}\label{fhat4}
		Let $u$ be an eigenform of $\Delta_i$ w.r.t. eigenvalue $\mu$, then for $|\Re(\l)|=1$,
		\[\|(\l-D_{t,T})^{-1}Q_Tu-\cE((\l-D_{t,i})^{-1}u)\|^2_{L^2}\leq \frac{C(|\l|+1+t^{-1})(1+t+t\mu^2+t\mu^4)\|u\|_{L^2}^2}{\sqrt{T}}.\]
	\end{lem}
	\begin{proof}
		Let $V_i=d^{Z_i}-d^{Z_i,*}$, then on $[-2,2]\times Y$, $V_i=V_i^\R+V_i^Y$, where $V_i^\R=ds\wedge\nabla_{\frac{\p}{\p s }} +i_{\frac{\p}{\p s}}\wedge\nabla_{\frac{\p}{\p s}}.$
		
		It follows from \eqref{qtdt} that
		\be\label{fhat421}
		(\l-D_{t,T})Q_T(\l-D_{t,i})^{-1}u|_{Z_i}=Q_T(u)|_{Z_i}.
		\ee
		Moreover, by trace formula and the fact that $\|(\l-D_{t,i})^{-1}\|\leq 1$, proceed as in Lemma \ref{fhat1},
		\be\label{fhat422}
		\int_{Y} |(\l-D_{t,i})^{-1}u_i|^2((-1)^i,y)+|\sqrt{t}V_i^{Y} (\l-D_{t,i})^{-1}u_i|^2((-1)^{i},y)\leq \frac{C(t\mu^2+t\mu^4+1+t)\|u\|^2_{L^2}}{\sqrt{T}} 
		\ee
		By (\ref{fhat421}), (\ref{fhat422}) and (\ref{qtdt}),
		\be\label{fhat42}\|(\l-D_{t,T})Q_T(\l-D_{t,i})^{-1}u-Q_T(u)\|^2_{L^2}\leq\frac{C(|\l|+1+t^{-1})(1+t+t\mu^2+t\mu^4)\|u\|_{L^2}^2}{\sqrt{T}}.\ee
		Since we also have $\|(\l-D_{t,T})^{-1}\|\leq1$ for $|\Re(\l)|=1$, (\ref{fhat42}) implies that
		\be\label{fhat43}\|Q_T(\l-D_{t,i})^{-1}u-(\l-D_{t,T})^{-1}Q_T(u)\|^2_{L^2}\leq\frac{C(|\l|+1+t^{-1})(1+t+t\mu^2+t\mu^4)\|u\|_{L^2}^2}{\sqrt{T}}.\ee
		By Lemma \ref{qtuu} and (\ref{fhat43}), the lemma follows.
		
	\end{proof}

	\begin{proof}[Proof of Theorem \ref{limitfhat}]
		
		Let $\{u_k\}_{k=1}^\infty$ be eigenforms of $\Delta_T$ such that $\{u_k\}$ forms an orthonormal basis.

		Fix $\epsilon>0$. By Lemma \ref{fhat1}, there exists $k_0=k_0(\epsilon,t)>0$, such that
		\be\label{fhat21}
		\sum_{k\geq k_0} \left|\left(N^{\bz}\left(f'(D_{t,T})-f'(\sqrt{t}V_T)\right)u_k,u_k\right)_{L^2}\right|<\epsilon.
		\ee

		By Lemma \ref{est1}, we may assume that 
		\begin{align}\begin{split}\label{epsilon}
				&\ \ \ \ \sum_{k\geq k_0}\left|\left(N^{\bz}f'(\sqrt{t}V_T)u_k,u_k\right)_{L^2}\right|\leq C \sum_{k\geq k_0} (1+\l_{k}^2(T,\s))e^{-t\l_{k}(T,\s)}<\epsilon.
		\end{split}\end{align}
		
		By (\ref{fhat21}) and (\ref{epsilon}),
		\be\label{fhat22}
		\sum_{k\geq k_0} \left|\left(N^{\bz}f'(D_{t,T})u_k,u_k\right)_{L^2}\right|<2\epsilon.
		\ee
		
		
		Similarly, one may assume that
		\be\label{fhat23}
		\sum_{k\geq k_0} \sum_{i=1}^2\left|\left(N^{Z_i}f'(D_{t,i})v_k,v_k\right)_{L^2}\right|<2\epsilon.
		\ee

		Let $\{v_k\}_{k=1}^{k_0}$ be orthonormal eigenforms with respect to eigenvalues $\{\l_k\}_{k=1}^{k_0}.$
		Let $E_{k_0}(\bz,\F)$ be the space generated by eigenforms with respect to eigenvalues $\l_1(T,\s),...\l_{k_0}(T,\s).$ Set $\P^{k_0}(T)$ be the orthogonal projection w.r.t. $E_{k_0}(\bz,\F).$
		Proceeding as in the proof of Propositon \ref{last2p}, one can see that $E_{k_0}(\bz,\F)$ is generated by $\{\P^{k_0}(T)Q_Tv_k\}_{k=1}^{k_0}$ if $T$ is large.
		
		Moreover,
		\be\label{epsilon3}
		\|\P^{k_0}(T)Q_Tv_k-Q_Tv_k\|^2_{L^2}\leq \frac{C(\l_{k_0}(T,\s)+1)}{\sqrt{T}},
		\ee
		\be\label{epsilon31}
		\|v_k-Q_Tv_k\|^2_{L^2}\leq \frac{C(\l_{k_0}(T,\s)+1)}{\sqrt{T}}.
		\ee

		Let $\{u_k(T)\}$ be the Gram-Schmidt Orthogonalization of $\{\P^{k_0}(T)Q_Tv_k\}_{k=1}^{k_0}$. Then Lemma \ref{qtuu}, \eqref{epsilon3} and \eqref{epsilon31} imply that

		\be\label{epsilon4}
		\|u_k(T)-Q_Tv_k\|^2_{L^2}\leq \frac{C(\l_{k_0}(T,\s)+1)}{\sqrt{T}},\ee
		
		Procceding as in the proof of \cite[Theorem 2.13]{bismut1995flat}, one can show that there exists $(T,\s)$-independent $C>0,$ such that
		\be\label{fhat32}\left\|\varphi f^{\prime}\left(D_{t,T}\right)-\P^{\delta}(T)\varphi f^{\prime}\left(\P^{\delta}(T)D^{\delta,T}_{t}\P^{\delta}(T)\right)\P^{\delta}(T)\right\|\leq \frac{C}{\sqrt{t}}.\ee
		(Comparing with Theorem \ref{intnew1}, we are looking at operator norm, instead of trace.)

		Let $f'_{\la}(D_{t,T}):=\varphi\frac{N^\bz}{2}\Big( f^{\prime}\left(D_{t,T}\right)-\P^{\delta}(T)f^{\prime}\big(\P^{\delta}(T)D^{\delta,T}_{t}\P^{\delta}(T)\big)\P^{\delta}(T)\Big).$
		
		Hence, by functional calculus, (\ref{epsilon4}), (\ref{fhat32}), Lemma \ref{o1} and Lemma \ref{fhat4}, there exists $T(\epsilon,t)>0$, such that when $T>T(\epsilon,t)$
		\begin{align}\begin{split}\label{epsilon5}
				&\ \ \ \ \|\sum_{k=1}^{k_0}\left(f'_{\la}(D_{t,T})u_k(T),u_k(T)\right)_{L^2}-A(t)\|_{L^\infty}\\
				&\leq \|\sum_{k=1}^{k_0}\left(f'_{\la}(D_{t,T})Q_Tv_k,Q_Tv_k\right)_{L^2}-A(t)\|_{L^\infty}+\epsilon \leq 2\epsilon. \\
		\end{split}\end{align}
		Here for simplicity, set $$A(t):=\sum_{k=1}^{k_0}\sum_{i=1}^2\left(\varphi\frac{N^{Z_i}}{2}\left( f^{\prime}\left(D_{t,i}\right)- \P_if^{\prime}(\P_iD_{t}^{\H_i}\P_i)\P_i \right)v_k,v_k\right)_{L^2}.$$ 
		By (\ref{fhat22}), (\ref{fhat23}) and (\ref{epsilon5}), 
		$$\lim_{T\to\infty}f_{\la}^{\wedge}\left(C_{t,T}^{\prime}, h^{\bE}\right)=\sum_{i=1}^2f_{\la}^{\wedge}\left(C_{t,i}^{\prime}, h^E\right).$$
		Similarly, one can show
		\[\lim_{T\to\infty}\sum_{i=1}^2f_{\la}^{\wedge}\left(\tC_{t,T,i}^{\prime}, h_T^\bE\right)=\sum_{i=1}^2f_{\la}^{\wedge}\left(C_{t,i}^{\prime}, h^E\right).\]
	\end{proof}
	\begin{proof}[Proof of Theorem \ref{larcon}]
		By Lemma \ref{est1}, Lemma \ref{fhat1} and Theorem \ref{eigencon}, proceeding as in the proof of \cite[Theorem 3.2]{Yantorsions}, one can see that there exists a measurable function $G(t)$ on $[1,\infty)$, s.t. $G(t)/t$ is $L^1([1,\infty)$-integrable ($G$ is independent of $T$ and $\s$). Moreover,
		\[|f_{\la}^{\wedge}\left(C_{t,T}^{\prime}, h^{\bar{E}}\right)|\leq G(t).\]
		
		By Theorem \ref{limitfhat} and the dominate convergence theorem,
		\begin{align*}&\ \ \ \ \lim_{T\to\infty} \mathcal{T}_{\la}^{\mL}\left(T^H \bm, g^{T \bz}, h^{\F}\right)(T)=\sum_{i=1}^2\mathcal{T}_{\la,i}^{\mL}\left(T^H M_i, g^{T Z_i}, h^{F_i}\right).\end{align*}
		Similarly,
		\[\lim_{T\to\infty}\sum_{i=1}^2\mathcal{T}_{\la,i}^{\mL}\left(T^H \bm_i, g^{T \bz_i}, h_T^{\F_i}\right)(T)=\sum_{i=1}^2\mathcal{T}_{\la,i}^{\mL}\left(T^H M_i, g^{T Z_i}, h^{F_i}\right).\\
		\]
	\end{proof}

	\section{The Small Time Contributions}\label{cons}
	\subsection{Several Hodge Laplacians}
	To show the gluing formula for $f^{\wedge}$, we introduce several Hodge Laplacians.
	
	Let $\Delta_{B,1}^\R$ be the Hodge Laplacian on $[-2,-1]$ with absolute boundary conditions. It's easy to see that $\ker(\Delta_{B,1}^\R)$ is one-dimensional and generated by constant functions. Thus, \be \label{tracer}\Tr_s((1+2\Delta_{B,1})e^{-t\Delta_{B,1}^\R})=\lim_{t\to\infty}\Tr_s((1+2\Delta_{B,1})e^{-t\Delta_{B,1}^\R})=1.\ee
	
	Let $\Delta_{B,2}^\R$ be the Hodge Laplacian on $[1,2]$ with relative boundary conditions. Similarly, \be \label{tracer1}\Tr_s((1+2\Delta_{B,2})e^{-t\Delta_{B,2}^\R})=-1.\ee
	
	
	Let $\bar{\Delta}_{B}$ be the Hodge laplacian on $\Omega([-2,2])$  satisfying the absolute boundary condition on $-2$, and relative boundary condition on $2$.

	We can also regards $p_T$ as a smooth function in $(-2,2)$, and let ${\Delta}^\R_T$ be the Witten Laplacian on $(-2,2)$ with respect to $p_T,$ with absolute boundary condition on $-2$, and relative boundary condition on $2$.

	\subsection{Gluing formulas for $f^{\wedge}\left(C_{t,i}^{\prime}, h^E\right)$ and $f^{\wedge}\left(C_{t,T}^{\prime}, h^{\bar{E}}\right)$}
	
	Let $D_{t,Y}:=D_{t}|_{Y}$.
	Let $\eta_i(i=1,2)$ be a smooth function on $(-\infty,\infty)$ satisfying
	\begin{enumerate}
		\item $0\leq\eta_i\leq1$;
		\item $\eta_1\equiv1$ in $(-\infty,-3/2)$,$\eta_1\equiv 0$ in $(-5/4,\infty)$;
		\item $\eta_2\equiv1$ in $(3/2,\infty)$,$\eta_2\equiv 0$ in $(-\infty,5/4)$.
	\end{enumerate}
	We can think $\eta_i$ as a function on $Z_i(i=1,2). $
	
	Let $\tf(a)=(1+2a)e^a$, then $\tf(a^2)=f'(a).$
	Proceeding as in \cite[\S 6]{Yantorsions} or \cite[\S 13(b)]{bismut1991complex}, since $T^HM,g^{TZ}$ and $h^F$ are porduct-type near $N$, for some $C,c>0$,
	\begin{align}\begin{split}\label{heatbd1}
			&\ \ \ \ \Big\|\sum_{i=1}^2\varphi \Tr_s\left(N^{Z} f'\left(D_{t,i}\right)\right)-\sum_{i=1}^2\varphi \Tr_s\left(N^{Z}\eta_i f'\left(D_{t,i}\right)\right)\\
			&-\sum_{i=1}^2\varphi \Tr_s\left(N^{Z} f'\left(D_{t,Y}\right)\otimes \tf(t\Delta^\R_{B,i})\right)\\
			&+\sum_{i=1}^2\varphi \Tr_s\left(N^{Z}\eta_i f'\left(D_{t,Y}\right)\otimes \tf(t\bar{\Delta}_{B})\right)\Big\|_{L^\infty}\leq C\exp(-c/t).
	\end{split}\end{align}
	Next, notice that $[-2,2]\times Y$, the number operator can be decomposed as $N^Z=N^{Y}+N^\R$ canonically (Here $N^Y$ and $N^\R$ are the number operator on $Y$ and $\R$ components respectively). By (\ref{tracer}), (\ref{tracer1}) and \cite[Theorem 3.15]{bismut1995flat},
	\begin{align}\begin{split}\label{heatbd2}
			&\ \ \ \ \sum_{i=1}^2\varphi \Tr_s\left(N^{Z} f'\left(D_{t,Y}\right)\otimes \tf(t\Delta^\R_{B,i})\right)\\
			&=\sum_{i=1}^2\varphi \Tr_s\left(N^{Y} f'\left(D_{t,Y}\right)\otimes \tf(t\Delta^\R_{B,i})\right)+\sum_{i=1}^2\varphi \Tr_s\left( f'\left(D_{t,Y}\right)\otimes N^{\R}\tf(t\Delta^\R_{B,i})\right).\\
			&=\sum_{i=1}^2\chi(Y)\rank(F)\Tr_s(N^{\R}\tf(t\Delta^\R_{B,i})).\\
	\end{split}\end{align}
	Similarly, for some $(T,\s)$-independent $C,c>0$,
	\begin{align}\begin{split}\label{heatm1}
			&\ \ \ \ \Big\|\varphi \Tr_s\left(N^{\bz} f'\left(D_{t,T}\right)\right)-\sum_{i=1}^2\varphi \Tr_s\left(N^{Z}\eta_i f'\left(D_{t,i}\right)\right)\\
			&-\varphi \Tr_s\left(N^{Z} f'\left(D_{t,Y}\right)\otimes \tf(t\Delta^\R_{T})\right)\\
			&+\sum_{i=1}^2\varphi \Tr_s\left(N^{Z}\eta_i f'\left(D_{t,Y}\right)\otimes \tf(t\bar{\Delta}_{B})\right)\Big\|_{L^\infty}
			\\ &\leq C\exp(-c/t).
	\end{split}\end{align}

	Moreover, $\Tr_s((1+2\Delta^{\R}_T)e^{-t{\Delta}_T^\R})=\lim_{t\to\infty}\Tr_s((1+2\Delta^{\R}_T)e^{-t{\Delta}_T^\R})=\dim(\ker({\Delta}_T^\R)_0)-\dim(\ker({\Delta}_T^\R)_1)$. Here $\ker({\Delta}_T^\R)_i$ denotes the space of harmonic $i$-forms$(i=0,1)$. Since $p_T$ is odd, one can see easily that if $u(s)\in\ker({\Delta}_T^\R)_0$, then $u(-s)ds\in\ker({\Delta}_T^\R)_1.$
	
	As a result, $\Tr_s((1+2\Delta^{\R}_T)e^{-t{\Delta}_T^\R})=0$.
	Proceeding as before,
	\begin{align}\begin{split}\label{heatm2}
			\ \ \ \ &\varphi \Tr_s\left(N^{Z} f'\left(D_{t,Y}\right)\otimes \tf(t\Delta^\R_{T})\right)=\chi(Y)\rank(F)\Tr_s(N^{\R}\tf(t\Delta_{T}^\R)).
	\end{split}\end{align}

	\subsection{Proof of Theorem \ref{main}}
	For a differential form $w$, let $w^0$ denote its degree $0$ component, and $w^+:=w-w^0.$
	
	Proceeding as in \cite[Proposition 2.18]{bismut1995flat}, for some $(T,\s)$-independent $C,$
	\begin{align}\begin{split}\label{eq100}
			&\ \ \ \ \left\|\varphi \operatorname{Tr}_s\left(\frac{N^Z}{2} \P^{\delta}(T)f^{\prime}\left(\P^{\delta}(T)D^{\delta,T}_{t}\P^{\delta}(T)\right)\P^{\delta}(T)\right)\right.\\
			&-\left.\sum_{i=1}^2\varphi \operatorname{Tr}_s\left(\frac{N^Z}{2} \P_if^{\prime}\left(\P_iD^{\H_i}_{t}\P_i\right)\P_i\right)\right\|_{L^\infty}\leq Ct.
	\end{split}\end{align}
	It follows from (\ref{heatbd1}), (\ref{heatbd2}), (\ref{heatm1}), (\ref{heatm2}), (\ref{eq100}), Theorem \ref{thmfhat} and dominated convergence theorem that
	
	\[\left(\mathcal{T}_{\la}^{\mS}\left(T^H \bm, g^{T \bz}, h^{\F}\right)(T)\right)^+=\sum_{i=1}^2\left(\mathcal{T}_{\la,i}^{\mS}\left(T^H M_i, g^{T Z_i}, h^{F_i}\right)\right)^++o(1).\]

	It follows from \cite[Theorem 3.3]{Yantorsions} that
	\begin{align*}&\ \ \ \ \left(\mathcal{T}_{\la}^{\mS}\left(T^H \bm, g^{T \bz}, h^{\F}\right)(T)\right)^0\\
		&=\sum_{i=1}^2\left(\mathcal{T}_{\la,i}^{\mS}\left(T^H M_i, g^{T Z_i}, h^{F_i}\right)\right)^0-(T-\log(2))\chi(Y)\rank(F)/2+o(1).\end{align*}
	
	Similarly, one can show that
	\[\left(\mathcal{T}_{\la,i}^{\mS}\left(T^H \bm_i, g^{T \bz_i}, h_T^{\F_i}\right)(T)\right)^+=\left(\mathcal{T}_{\la,i}^{\mS}\left(T^H M_i, g^{T Z_i}, h^{F_i}\right)\right)^++o(1),\]
	and 
	\begin{align*}&\ \ \ \ \left(\mathcal{T}_{\la}^{\mS}\left(T^H \bm, g^{T \bz}, h^{\F}\right)(T)\right)^0\\
		&=\sum_{i=1}^2\left(\mathcal{T}_{\la,i}^{\mS}\left(T^H M_i, g^{T Z_i}, h^{F_i}\right)\right)^0-T\chi(Y)\rank(F)/4+o(1).\end{align*}

	\section{Small Eigenvalues and Mayer-Vietoris Sequences}\label{lastreal}
	From now on, we assume that $T>0$ is large enough.
	
	
	\subsection{Some estimate of harmonic forms on the tube}
	Let $w\in\H(\bm_i,\F_i)(T)$ such that $\|w\|_{L^2,T}=1$.
	Notice that $\tilde{\Delta}_T=e^{p_T}\Delta_Te^{-p_T}$. By \cite[Lemma 4.3]{Yantorsions},
	\begin{prop}\label{agmon1}
		For $w\in \H(\bm_1,\F_1)(T)$, any $l\in\Z^+$, there exists $(T,\s)$-independent constant $C_l$, s.t.
		\[\int_{-1/2}^0\int_Y|e^{-p_T}w|^2ds\dvol_Y\leq \frac{C_l}{T^l};\]
		for $w\in \H(\bm_2,\F_2)(T)$,
		\[\int_{0}^{1/2}\int_Y|e^{-p_T}w|^2ds\dvol_Y\leq \frac{C_l}{T^l}.\]
	\end{prop}
	\begin{proof}
		\def\bw{\bar{w}}
		Let $\bw=e^{-p_T}w$, then $\|\bw\|_{L^2}=1$, and $\Delta_T\bw=0.$
		We may as well assume that $w\in\H(\bm_1,\F_1)(T).$
		Let $\eta_1$ be a nonnegative bounded smooth function on $\bm_1$, such that $\eta_1|_{[-\frac{1}{2},0]\times Y}\equiv1$, $\eta_1|_{\bm_1-[-3/4,0]\times Y}\equiv 0$, $|\nabla \eta_1|\leq 64$. Let $\bw=\a+\b ds$, then
		by Lemma \ref{limit1} and integration by parts and notice that $\p_sp_T\geq 0$,
		\begin{align}\begin{split}\label{new12}
				&\ \ \ \ 0= \int_{\bz_1}\lan\Delta_T\bw,\eta_1^2\bw\ran\dvol_{\bz_1}\\
				&=\int_{\bz_1}\lan D^{\bz}\bw,D^{\bz}\eta_1^2\bw\ran+\lan L_{T}\bw,\eta^2_1\bw\ran+|\nabla p_T|^2\eta_1^2|\bw|^2\dvol_{\bz_1}\\
				&-\int_Y \lan ds\wedge \bw, d^{\bz}\bw\ran(0,y)\dvol_Y\\
				&=\int_{\bz_1}\lan D^{\bz}\bw,D^{\bz}\eta_1^2\bw\ran+\lan L_{T}\bw,\eta^2_1\bw\ran+|\nabla p_T|^2\eta_1^2|\bw|^2\dvol_{\bz_1}\\
				&-\int_Y \lan ds\wedge \bw,d_T\bw\ran(0,y)\dvol_Y+\int_Y\partial_s p_T\lan\a,\a\ran(0,y)\dvol_Y\\
				&\geq \int_{\bz_1}\lan D^{\bz}\bw,D^{\bz}\eta_1^2\bw\ran+\lan L_{T}\bw,\eta^2_1\bw\ran+|\nabla p_T|^2\eta_1^2|\bw|^2\dvol_{\bz_1}\\
				&\geq\int_{\bz_1}\eta_1^2\lan D^{\bz}\bw,D^{\bz}\bw\ran-|\eta_1'\eta_1||\lan D^{\bz}\bw,\bw\ran|+\lan L_{T}\bw,\eta_1^2\bw\ran+|\nabla p_T|^2\eta_1^2|\bw|^2\dvol_{\bm}\\
				&\geq\int_{\bz_1}\eta_1^2\lan D^{\bz}\bw,D^{\bz}\bw\ran/2-4|\eta_1'|^2|\bw|^2+C'T^2\eta_1^2|\bw|^2\dvol_{\bz_1},\\
		\end{split} \end{align}
		where $L_T$ is some linear operator, such that restricted on $[-1,0]\times Y$, $L_T\bw=-p_T''(s)\a+p_T''\b ds.$ 
		(\ref{new12}) implies that
		\be \label{new1211}
		\int_{-1/2}^{0}\int_Y|\bw|^2\dvol_Yds\leq \frac{C\int_{-3/4}^{0}\int_Y|\bw|^2\dvol_Yds}{T^2}\leq \frac{C}{T^2}.
		\ee
		
		Similarly, one can show that 
		\be\label{new1212}
		\int_{-3/4}^{0}\int_Y|\bw|^2\dvol_Yds\leq \frac{C\int_{-7/8}^{0}\int_Y|\bw|^2\dvol_Yds}{T^2}.
		\ee
		It follows from (\ref{new1211}) and (\ref{new1212}) that
		\be\label{new1213}
		\int_{-1/2}^{0}\int_Y|\bw|^2\dvol_Yds\leq \frac{C'\int_{-7/8}^{0}\int_Y|\bw|^2\dvol_Yds}{T^4}\leq \frac{C'}{T^4}.
		\ee
		Keep on doing this, one has Proposition \ref{agmon1} for all $l\in\Z^+.$
		
	\end{proof}

	\begin{prop}\label{agmon2}
		For any $l\in\Z^+,$
		\[\|\tP^{\delta}(T)\cE(w)-\cE(w)\|_{L^2,T}\leq \frac{C_l}{T^l}.\]
	\end{prop}
	\begin{proof}
		
		We may as well assume that $w\in \H(\bz_1,\F_1)(T)$. Let $\eta\in C_c^\infty(\R)$, s.t. $\eta(s)=1$ if $s\in(-\infty,-1/8)$ and $\eta|_{[-1/16,0]}\equiv0$, we can treat $\eta$ as a smooth function on $\bz.$
		
		By Proposition \ref{agmon1},
		\be\label{agmon11}\|w-\eta w\|_{L^2,T}\leq\frac{C_l}{T^l}\ee
		and
		\be\label{agmon12}\|\tilde{D}_T\eta w\|_{L^2,T}\leq\left\||\eta'||w|\right\|_{L^2,T}\leq\frac{C_l}{T^l}.\ee

		Proceeding as in the proof of Proposition \ref{last1p}, the proposition follows.

	\end{proof}

	By Lemma \ref{limit1}, 
	\begin{lem}\label{agmon3}
		When $|s-(-1)^{i}|\leq \frac{2}{\sqrt{T}}$, $\int_Y|e^{-p_T}w|^2(s,y)\dvol_Y\leq C.$
	\end{lem}
	
	\begin{lem}\label{agmon4}
		For $w\in\H(\bz_1,\F_1)(T)$,
		\[|\int_{-1}^{-1/16}\int_Y(s+1)^2|e^{-p_T}w|^2\dvol_Yds|\leq\frac{C}{T^{3/2}}.\]
		For $w\in\H(\bm_2,\F_2)(T)$,
		\[|\int^{1}_{1/16}\int_Y(s-1)^2|e^{-p_T}w|^2\dvol_Yds|\leq\frac{C}{T^{3/2}}.\]
	\end{lem}
	\begin{proof}
		\def\bw{\bar{w}}
		\def\bet{\bar{\eta}}
		WLOG, assume $w\in\H(\bz_1,\F_1)(T)$. Let $\bw=e^{-p_T}w$ and $\bar{w}=\a+\b ds$ in $[-2,0]\times Y.$ Let $\bet$ be a smooth function on $(-\infty,0]$, such that $\bet|_{(-\infty,-1/16)}\equiv1$, $\bet|_{[-1/32,0]}\equiv0$. We could regard $\bet$ as a smooth function on $\bz_1.$

		Since $\Delta_T e^{-p_T}w=0$, as in the proof of Lemma \ref{agmon1}, one can see that
		\begin{align}\begin{split}\label{new1201}
				&\ \ \ \ 0= \int_{\bz_1}\lan\Delta_T\bw,\bet^2\bw\ran\dvol_{\bz_1}\\
				&\geq\int_{\bz_1}\bet^2\lan D^{\bz}\bw,D^{\bz}\bw\ran/2-4|\bet'|^2||\bw|^2+\lan L_{T}\bw,\bet^2\bw\ran+|\nabla p_T|^2\bet^2|\bw|^2\dvol_{\bz}\\
		\end{split} \end{align}
		where $L_T$ is some linear operator, such that restricted on $[-1,0]\times Y$, $L_T\bw=-p_T''(s)\a+p_T''\b ds.$ 
		
		Note that on $[-1+\frac{2}{\sqrt{T}},0]$, $|p_T'(s)|^2-|p_T''|\geq 1/2|p_T'(s)|^2$, by Lemma \ref{agmon3} and (\ref{new1201}),
		\begin{align*}
			&\ \ \ \ |\int_{-1+\frac{2}{\sqrt{T}}}^{-1/16}\int_YT^2(s+1)^2|e^{-p_T}w|^2\dvol_Yds|\\
			&\leq C\int_{-1}^{-1+\frac{2}{\sqrt{T}}}\int_YT|e^{-p_T}w|^2\dvol_Y ds+C\int_{\bz_1}|\bet'|^2|w|^2\dvol\leq C\sqrt{T}.
		\end{align*}
		That is,
		\be\label{agmon42}|\int_{-1+\frac{2}{\sqrt{T}}}^{-1/16}\int_Y(s+1)^2|e^{-p_T}w|^2\dvol_Yds|\leq\frac{C}{T^{3/2}}.\ee
		It follows from Lemma \ref{agmon3} that
		\be\label{agmon43}|\int_{-1}^{-1+\frac{2}{\sqrt{T}}}\int_Y(s+1)^2|e^{-p_T}w|^2\dvol_Yds|\leq\frac{C}{T^{3/2}}.\ee
		The lemma then follows from (\ref{agmon42}) and (\ref{agmon43}).
	\end{proof}
	
	\subsection{Estimate of small eigenvalues}

	\begin{lem}\label{last20}
		When $T$ is big enough, $\|\frac{\p}{\p T}\P^{\delta}(T)\|\leq C$ for some $(T,\s)$-independent $C>0.$ Moreover, there exists a uniformly bounded operator $U_T,$ such that 
		\be\label{last210}\frac{\p}{\p T}\P^{\delta}(T)=[\bar{D}^{\bz}_T,U_T].\ee
		Here $\bar{D}^{\bz}_T:=d^{\bz}_T-d^{\bz,*}_T.$ 
		Similar statements hold if we replace $\P^{\delta}(T)$ by $\tP^{\delta}(T)$, $\tP_i(T)$ e.t.c.
	\end{lem}
	\begin{proof}
		Let $\gamma$ be the circle of radius $\sqrt{3\delta/2}$ with center $0$, and oriented positively. Then
		\[\P^{\delta}(T)=\int_{\gamma}(\l-D^{\bz}_T)^{-1}d\l.\]
		As a result, 
		\begin{align*}\frac{\p}{\p T}\P^{\delta}(T)=\int_{\gamma}(\l-D^{\bz}_T)^{-1}\frac{\p}{\p T}D^{\bz}_T(\l-D^{\bz}_T)^{-1}d\l.\end{align*}
		Now, one can check easily that when $T$ is large enough, $\|(\l-D^{\bz}_T)^{-1}\|\leq\left((\sqrt{3/2}-1)\delta\right)^{-1},$ and 
		\[|\frac{\p}{\p T}\frac{\p}{\p s}p_T(s)|\leq C.\]
		
		Thus, $\|\frac{\p}{\p T}\P^{\delta}(T)\|\leq C$.
		
		Notice that $\frac{\p}{\p T}D^{\bz}_T=[\bar{D}^{\bz}_T,\frac{\p}{\p T}p_T]$, $|\frac{\p}{\p T}p_T|\leq C$ and $\bar{D}^{\bz}_T$ commutes with $D^{\bz}_T$, one has (\ref{last210}) for
		\[U_T=\int_{\gamma}(\l-D^{\bz}_T)^{-1}\frac{\p}{\p T}p_T(\l-D^{\bz}_T)^{-1}d\l.\]
	\end{proof}

	When $T$ is large enough, 
	\be\label{identify0}\tO^k_{\sm}(\bz,\F)(T)=\te_{k,T} \H(\bz_2,\F_2)(T)\oplus\tir^*_{k,T}\H(\bz_1,\F_1)(T).\ee
	Let $\te_{k,T}^{-1}$ be the inverse of $\te_{k,T}|_{\te_{k,T} \H(\bz_2,\F_2)(T)}$, and $\tir^{-1}_{k,T}$ be the inverse of $\tir_{k,T}|_{\tir^*_{k,T}\H(\bz_1,\F_1)(T)}.$
	
	Next, we will put a family of metric $g_T$ on $H_{\abs}(\bz_1,\F_1)\oplus H_{\rel}(\bz_2,\F_2)$ when $T$ is large enough.
	
	First, we define a map $R_T:H^k_{\abs}(\bz_2,\F_2)\oplus H_{\rel}^k(\bz_1,\F_1)\to\tO_{\sm}(\bm,\F)$ as follows.
	
	For $[u]\in H^k_{\abs}(\bz_1,\F_1)$ represented by $u\in\Omega^k_{\abs}(\bz_1,\F_1)$,
	set $R_T([u]):=\te_{k,T}\tP_1(u)=\tP^{\delta}(T)\cE(\tP_1(T)u).$  For $[v]\in H^k_{\rel}(\bz_2,\F_2)$ represented by $v\in\Omega^k_{\rel}(\bz_2,\F_2)$,
	set $R_T([v]):=\tir_{k,T}^{-1}\tP_2(v).$
	
	Then set $g_T(w,w):=(R_Tw,R_Tw)_{L^2,T}$ for $w\in H^k_{\abs}(\bz_1,\F_1)\oplus H^k_{\rel}(\bz_2,\F_2)$.

	\begin{lem}\label{last30}
		There exists $(T,\s)$-independent $C>0$, such that
		$1-\frac{C}{T^{3/2}}\leq\|g_T^{-1}\frac{\p}{\p T}g_T\|\leq 1+\frac{C}{T^{3/2}}.$ Here the operator norm is taken with respect to $g_T.$
	\end{lem}
	\begin{proof}
		In the following, some $\tP_2(T)u$ should be understood as $\cE(\tP_2(T)u).$
		First, it's clear that for a family of projection operators $P(T)$, 
		\be\label{last31}P(T)\frac{\p}{\p T} P(T) P(T)=0.\ee
		
		For any $[u],[v]\in H^k_{\rel}(\bz_2,\F_2)$, 
		\begin{align}\begin{split}\label{last32}
				&\ \ \ \ \frac{\p}{\p T}g_T([u],[v])=\frac{\p}{\p T}(R_T[u],R_T[v])_{L^2,T}=\frac{\p}{\p T}(\tP^{\delta}(T)\tP_2(T)u,\tP^{\delta}(T)\tP_2(T)v)_{L^2,T}\\
				&=(\frac{\p}{\p T}\tP^{\delta}(T)\tP_2(T)u,\tP^{\delta}(T)\tP_2(T)v)_{L^2,T}+(\tP^{\delta}(T)\tP_2(T)u,\pT\tP^{\delta}(T)\tP_2(T)v)_{L^2,T}\\
				&+(\tP^{\delta}(T)\pT\tP_2(T)u,\tP^{\delta}(T)\tP_2(T)v)_{L^2,T}+(\tP^{\delta}(T)\tP_2(T)u,\tP^{\delta}(T)\pT\tP_2(T)v)_{L^2,T}\\
				&-2(\pT p_T\tP^{\delta}(T)\tP_2(T)u,\tP^{\delta}(T)\tP_2(T)v)_{L^2,T}.
		\end{split}\end{align}

		Let $0\leq\eta\in C_c^\infty(\R)$, s.t. $\eta(s)=1$ if $s\in(1/8,\infty)$ and $\eta|_{[0,1/16]}\equiv0$, we can treat $\eta$ as a smooth function on $\bm.$
		
		By Proposition \ref{agmon2}, (\ref{agmon11}), (\ref{agmon12}) and Lemma \ref{last20},
		\begin{align}\begin{split}\label{last33}
				&\ \ \ \ (\frac{\p}{\p T}\tP^{\delta}(T)\tP_2(T)u,\tP^{\delta}(T)\tP_2(T)v)_{L^2,T}\\
				&\leq \frac{C}{T^2}\|\tP_2(T)u\|_{L^2,T}\|\tP_2(T)v\|_{L^2,T}+(\frac{\p}{\p T}\tP^{\delta}(T)\eta\tP_2(T)u,\eta\tP_2(T)v)_{L^2,T}\\
				&\leq \frac{C}{T^2}\|\tP_2(T)u\|_{L^2,T}\|\tP_2(T)v\|_{L^2,T}+([\bar{D}^{\bz}_T,U_T]\eta\tP_2(T)u,\eta\tP_2(T)v)_{L^2,T}\\
				&\leq \frac{C}{T^2}\|\tP_2(T)u\|_{L^2,T}\|\tP_2(T)v\|_{L^2,T}\\
				&+(U_T\bar{D}^{\bz}_T\eta\tP_2(T)u,\eta\tP_2(T)v)_{L^2,T}+(\eta\tP_2(T)u,U_T\bar{D}^{\bz}_T\eta\tP_2(T)v)_{L^2,T}\\
				&\leq \frac{C'}{T^2}\|\tP_2(T)u\|_{L^2,T}\|\tP_2(T)v\|_{L^2,T}.\\
		\end{split}\end{align}
		
		By Hodge theory, one can see that if $T'\geq T,$ then $\tP_{2}(T')\tP_2(T)=\tP_{2}(T')$, hence 
		\be\label{last35}
		\frac{\p}{\p T}\tP_2(T)=\pT\tP_2(T)\tP_2(T).
		\ee
		By (\ref{last31}), (\ref{last35}), Proposition \ref{agmon2} and Lemma \ref{last20},
		\begin{align}\begin{split}\label{last34}
				&\ \ \ \ (\tP^{\delta}(T)\pT\tP_2(T)u,\tP^{\delta}(T)\tP_2(T)v)_{L^2,T}\\
				& \leq \frac{C}{T^2}\|\tP_2(T)u\|_{L^2,T}\|\tP_2(T)v\|_{L^2,T}+ (\pT\tP_2(T)\tP_2(T)u,\tP_2(T)v)_{L^2,T}\\
				& =\frac{C}{T^2}\|\tP_2(T)u\|_{L^2,T}\|\tP_2(T)v\|_{L^2,T}+ (\tP_2(T)\pT\tP_2(T)\tP_2(T)u,v)_{L^2,T}\\
				& =\frac{C}{T^2}\|\tP_2(T)u\|_{L^2,T}\|\tP_2(T)v\|_{L^2,T}.
		\end{split}\end{align}
		By Proposition \ref{agmon2}, Lemma \ref{agmon3} and Lemma \ref{agmon4} and the construction of $p_T$,
		\begin{align}\begin{split}\label{last36}
				&\ \ \ \ 2(\pT p_T\tP^{\delta}(T)\tP_2(T)u,\tP^{\delta}(T)\tP_2(T)v)_{L^2,T}\\
				&\leq \frac{C}{T^2}\|\tP_2(T)u\|_{L^2,T}\|\tP_2(T)v\|_{L^2,T}+2(\pT p_T\tP_2(T)u,\tP_2(T)v)_{L^2,T}\\
				&\leq \frac{C}{T^{3/2}}\|\tP_2(T)u\|_{L^2,T}\|\tP_2(T)v\|_{L^2,T}+(\tP_2(T)u,\tP_2(T)v)_{L^2,T}.\\
		\end{split}\end{align}
		It follows (\ref{last32}), (\ref{last33}), (\ref{last34}), (\ref{last36}) that
		for any $[u],[v]\in H^k_{\rel}(\bz_2,\F_2)$, 
		\begin{align}\begin{split}\label{last361}
				\frac{\p}{\p T}g_T([u],[v])\leq \frac{C}{T^{3/2}}\|\tP_2(T)u\|_{L^2,T}\|\tP_2(T)v\|_{L^2,T}+(\tP_2(T)u,\tP_2(T)v)_{L^2,T}.
		\end{split}\end{align}
		Similarly,
		for any $[u],[v]\in H^k_{\rel}(\bz_2,\F_2)$, 
		\begin{align}\begin{split}\label{last362}
				\frac{\p}{\p T}g_T([u],[v])\geq -\frac{C}{T^{3/2}}\|\tP_2(T)u\|_{L^2,T}\|\tP_2(T)v\|_{L^2,T}+(\tP_2(T)u,\tP_2(T)v)_{L^2,T}.
		\end{split}\end{align}
		By Proposition \ref{agmon2}, proceeding as in Proposition \ref{last4p1}, one can see that
		\[1-\frac{C}{T^2}\leq\|\tir_{k,T}^*\|\leq 1+\frac{C}{T^2}.\]
		Similarly, for any $[u],[v]\in H^k_{\abs}(\bz_1,\F_1)$, 
		\begin{align}\begin{split}\label{last363}
				\frac{\p}{\p T}g_T([u],[v])\leq \frac{C}{T^{3/2}}\|\tP_1(T)u\|_{L^2,T}\|\tP_1(T)v\|_{L^2,T}-(\tP_1(T)u,\tP_1(T)v)_{L^2,T},
		\end{split}\end{align}
		\begin{align}\begin{split}\label{last364}
				\frac{\p}{\p T}g_T([u],[v])\geq -\frac{C}{T^{3/2}}\|\tP_1(T)u\|_{L^2,T}\|\tP_1(T)v\|_{L^2,T}-(\tP_1(T)u,\tP_1(T)v)_{L^2,T};
		\end{split}\end{align}
		for any $[u]\in H^k_{\abs}(\bz_1,\F_1),[v]\in H^k_{\rel}(\bz_2,\F_2)$, 
		\begin{align}\begin{split}\label{last3631}|\frac{\p}{\p T}g_T([u],[v])|\leq \frac{C}{T^{3/2}}\|\tP_1(T)u\|_{L^2,T}\|\tP_2(T)v\|_{L^2,T}.
		\end{split}\end{align}
		Lastly, by Proposition \ref{agmon1}, for any $[u],[v]\in H^{k}_{\abs}(\bz_1,\F_1)\oplus H^{k}_{\rel}(\bz_2,\F_2)$,
		\be\label{last37}
		|g_T([u],[v])-(\tP_{i}(T)u,\tP_{i}(T)v)_{L^2,T}|\leq \frac{C}{T^2}\|\tP_i(T)u\|_{L^2,T}\|\tP_i(T)v\|_{L^2,T}.
		\ee
		It follows (\ref{last361}), (\ref{last362}), (\ref{last363}), (\ref{last364}), (\ref{last3631}) and (\ref{last37}) that
		\[1-\frac{C}{T^{3/2}}\leq\|g_T^{-1}\pT g_T \|\leq1+ \frac{C}{T^{3/2}}.\]
	\end{proof}
	
	Next, we define a differential $\tp: H^{k-1}_{\abs}(\bz_1,\F_1)\oplus H_{\rel}^{k-1}(\bz_2,\F_2)\to H^k_{\abs}(\bz_1,\F_1)\oplus H_{\rel}^k(\bz_2,\F_2)$, $([u],[v])\mapsto(0,\p_k[u]).$ Recall that $\p_k$ is the map in  Mayer-Vietoris  sequence (\ref{mv}). Then one can check easily that
	$R_T \tp R_T^{-1}=d^{\bz}|_{\tO_{\sm}(\bm,\F)}$ and $R_T \tp_T^* R_T^{-1}=\delta_T^{\bz,*}|_{\tO_{\sm}(\bm,\F)}$, where $\tp^*_T$ is the dual of $\tp$ w.r.t. $g_T.$
	
	It follows from the statement in Step 1 in the proof of \cite[ Theorem 3.8]{10.2140/apde.2021.14.77} and Lemma \ref{last30} that
	\begin{cor}\label{last40}\label{smalleigen}
		When $T$ is large enough, all nonzero eigenvalues of $\Delta_T$ inside $[0,\delta]$ are actually inside $[c_1^2e^{-2T},c^2_2e^{-2T}]$ for some $c_2>c_1>0.$
	\end{cor}
	\subsection{Comparison of connections}
	$\H(\bz_i,\F_i)(T)$ has a flat connection $\nabla^{\H_i,T}:=\tP_i(T)\nabla^{\bE_i}$(c.f.  \cite[Proposition 2.6]{bismut1995flat}). Recall that if $T$ is large enough,
	\be\label{identify}\tO^k_{\sm}(\bz,\F)(T)=\te_{k,T} \H^k(\bz_2,\F_2)(T)\oplus\tir^*_{k,T}\H^k(\bz_1,\F_1)(T).\ee Recall that $\te_{k,T}^{-1}$ is the inverse of $\te_{k,T}|_{\te_{k,T} \H(\bz_1,\F_1)(T)}$, and $\tir^{-1}_{k,T}$ is the inverse of $\tir_{k,T}|_{\tir^*_{k,T}\H(\bz_2,\F_2)(T)}$, then $\tO_{\sm}(\bz,\F)(T)$ has a flat connection $$\nabla^{\H,T}:=\te_{k,T}\nabla^{\H_2,T}\te_{k,T}^{-1}\oplus\tir_{k,T}^{-1}\nabla^{\H_1,T}\tir_{k,T}.$$
	Moreover, let $\nabla^{\delta,T}:=\tP^{\delta}(T)\nabla^{\bE}$, then $\nabla^{\delta,T}$ is another connection on $\tO_{\sm}(\bz,\F)(T).$
	
	In this subsection, we are going to compare $\nabla^{\H,T}$ and $\nabla^{\delta,T}.$
	
	First, one has
	\begin{lem}\label{last50}
		When $T$ is big enough, $\|[\nabla^{\bE},\tP^{\delta}(T)]\|\leq C$ for some $(T,\s)$-independent $C>0$. Moreover, there exists a uniformly bounded operator  valued $1$-form $A_T,$ such that 
		\be\label{last21}[\nabla^{\bE},\tP^{\delta}(T)]=[d^{\bz},A_T].\ee
		Similar statements hold if we replace $\tP^{\delta}(T)$ by $\tP_i(T)$ ($i=1,2$).
	\end{lem}
	\begin{proof}
		Let $\gamma$ be the circle of radius $\sqrt{3\delta/2}$ with center $0$, and oriented positively. Since $[\nabla^E,(\l-D^{\bz}_T)^{-1}]=(\l-D^{\bz}_T)^{-1}[\nabla^E,D^{\bz}_T](\l-D^{\bz}_T)^{-1}$,
		\[[\nabla^E,\P^{\delta}(T)]=\int_{\gamma}(\l-D^{\bz}_T)^{-1}[\nabla^E,D^{\bz}_T](\l-D^{\bz}_T)^{-1}d\l.\]
		Since $g^{T\bz}$, $T^H\bm$ and $h^{\F}_T$ are product type near $N$, $\|[\nabla^E,D^{\bz}_T]\|\leq C$ for some $(T,\s)$-independent constant $C$. Proceeding  as in Lemma \ref{last20}, one has $\|[\nabla^E,\P^{\delta}(T)]\|\leq C$.
		
		Since $[d^{\bz},\nabla^{\bE}]=0$, one can see that
		\[[\nabla^{\bE},\tilde{\Delta}_T]=[d^{\bz},[\nabla^{\bE},\delta_T^{\bz,*}]].\]
		Since $g^{T\bz}$, $T^H\bm$ and $h^{\F}_T$ are product type near $N$, one can see that $[\nabla^{\bE},\delta_T^{\bz,*}]$ is uniformly bounded. Doing functional calculus for $\tilde{\Delta}_T$ as above, the lemma follows.
	\end{proof}
	
	Let $K^\delta_{T}=\nabla^{\delta,T}-\nabla^{\delta,T,*}, K^{\H}_{T}=\nabla^{\H,T}-\nabla^{\H,T,*}$ , $K_{T}=K^{\delta}_{T}-K^{\H}_{T}.$ 
	\begin{lem}\label{last60}
		When $T$ is large enough, $\|K_T\|\leq \frac{C}{T^{2}}$ for some $(T,\s)$-independent $C>0.$
	\end{lem}
	\begin{proof}
		For $u\in\te_{k,T} \H^k(\bz_2,\F_2)(T),$ there exists $v\in  \H^k(\bz_2,\F_2)(T)$, such that $u=\te_{k,T}v$. 
		Then 
		\begin{align}\begin{split}\label{last61}
				&\ \ \ \ \nabla^{\delta,T}u=\tP^{\delta}(T)\nabla^{\bE}\tP^{\delta}(T)\cE(v)=\tP^{\delta}(T)\nabla^{\bE}\cE(v)+\tP^{\delta}(T)[\tP^{\delta}(T),\nabla^{\bE}]\cE(v).
			\end{split}
		\end{align}
		Recall that $\cE(v)$ is an extension of $v$, s.t. outside $\bz_2$, $\cE(v)=0.$ Integration by parts as in the proof of Theorem \ref{exathm} shows that $\cE(v)$ is $d^{\bz}$-closed. Hence, by Lemma \ref{last50} and Corollary \ref{last40},
		
		\begin{align}\begin{split}\label{last62}
				&\ \ \ \ \|\tP^{\delta}(T)[\tP^{\delta}(T),\nabla^{\bE}]\cE(v)\|=\|\tP^{\delta}(T)d^{\bz}A_T\cE(v)\|\\
				&=\|d^{\bz}\tP^{\delta}(T)A_T\cE(v)\|\leq Ce^{-T}\|v\|.
			\end{split}
		\end{align}
		
		Similarly,
		\begin{align}\begin{split}\label{last63}
				&\ \ \ \ \nabla^{\H,T}u=\tP^{\delta}(T)\cE(\tP_{2}(T)\nabla^{\bE_2}v)=\tP^{\delta}(T)\cE(\nabla^{\bE_2}v)-\tP^{\delta}(T)\cE([\tP_2(T),\nabla^{\bE_2}]v)\\&=\tP^{\delta}(T)\nabla^{\bE}\cE(v)-\tP^{\delta}(T)\cE([\tP_2(T),\nabla^{\bE_2}]v).
			\end{split}
		\end{align}
		
		Moreover, Integration by parts as in the proof of Theorem \ref{exathm} shows that for $w\in\Omega_{\rel}(\bz_2,\F_2)(T)$,
		\be\label{last631}
		d^{\bz}\cE(w)=\cE(d^{\bz_2}w).
		\ee
		
		Thus, by (\ref{last63}),(\ref{last631}), Lemma \ref{last50}, Corollary \ref{last40} and a similar argument above,
		\begin{align}\begin{split}\label{last64}
				&\ \ \ \ \|\tP^{\delta}(T)\cE([\tP_2(T),\nabla^{\bE_2}]v)\|=\|\P^{\delta}(T)d^{\bz}\cE(A_{T,2}v)\|\\
				&=\|d^{\bz}\P^{\delta}(T)\cE(A_{T,2}v)\|\leq Ce^{-T}\|v\|.
			\end{split}
		\end{align}
		It follows from (\ref{last61}), (\ref{last62}), (\ref{last63}) and (\ref{last64}) and Proposition \ref{agmon2} that
		\be\label{last641}\|(\nabla^{\delta,T}-\nabla^{\H,T})u\|\leq Ce^{-T}\|u\|.\ee
		While for any $u_1,u_2\in \te_{k,T} \H^k(\bz_2,\F_2)(T)$,
		\be\label{last6411}(K_Tu_1,u_2)_{L^2,T}=((\nabla^{\delta,T}-\nabla^{\H,T})u_1,u_2)_{L^2,T}+(u_1,(\nabla^{\delta,T}-\nabla^{\H,T})u_2)_{L^2,T}.\ee
		Hence, by (\ref{last641}) and (\ref{last6411})
		\be\label{last642}|(K_Tu_1,u_2)_{L^2,T}|\leq Ce^{-T}\|u_1\|_{L^2,T}\|u_2\|_{L^2,T}.\ee
		Similarly, one can show that restricted on $\tir^*_{k,T}\H(\bz_1,\F_1)(T)$,
		$\|\nabla^{\delta,T,*}-\nabla^{\H,T,*}\|\leq Ce^{-T}.$ 
		Similarly, for $u_1,u_2\in\tir_{k,T}^*\H(\bz_1,\F_1)(T)$,
		\be\label{last6422}|(K_Tu_1,u_2)_{L^2,T}|\leq Ce^{-T}\|u_1\|_{L^2,T}\|u_2\|_{L^2,T}.\ee
		Since $T^H\bm,g^{T\bz}$ and $h^{\F}$ are product-type near $N$, 
		\be\label{last643}\|\nabla^{\bE}-\nabla^{\bE,*}\|\leq C.\ee
		Suppose $v_i\in\H(\bz_i,\F_i)(T)$, $i=1,2$, and $u_1=\tir^*_{k,T}v_1$, $u_2=\te_{k,T}v_2$. Let $\eta\in C_c^\infty(\R)$, s.t. $\eta(s)=1$ if $s\in(-\infty,-1/8)$ and $\eta|_{[-1/16,0]}\equiv0.$ Then we could think $\eta$ as a function on $\bz.$
		
		By Proposition \ref{agmon1}, \ref{agmon2}, (\ref{rstar}) and (\ref{last643}) , 
		\begin{align}\begin{split}\label{last644}&\ \ \ \ ((\nabla^{\delta,T}-\nabla^{\delta,T,*})u_1,u_2)_{L^2,T}=((\nabla^{\delta,T}-\nabla^{\delta,T,*})\tP^{\delta}(T)\cE(v_1),\tP^{\delta}(T)\cE(v_2))_{L^2,T}\\&=((\nabla^{\bE}-\nabla^{\bE,*})\tP^{\delta}(T)\cE(v_1),\tP^{\delta}(T)\cE(v_2))_{L^2,T}\\
				&\leq \frac{C}{T^2}\|u_1\|_{L^2,T}\|u_2\|_{L^2,T}+((\nabla^{\bE}-\nabla^{\bE,*})\eta \cE(v_1),\cE(v_2))_{L^2,T}\\
				&= \frac{C}{T^2}\|u_1\|_{L^2,T}\|u_2\|_{L^2,T}+(\eta(\nabla^{\bE}-\nabla^{\bE,*}) \cE(v_1),\cE(v_2))_{L^2,T}\\
				&=\frac{C}{T^2}\|u_1\|_{L^2,T}\|u_2\|_{L^2,T}.\\
		\end{split}\end{align}
		Since for any $u_1\in\tir_{k,T}^*\H(\bz_1,\F_1)(T),u_2\in\te_{k,T}\H(\bz_2,\F_2)(T)$, 
		\be\label{last645}((\nabla^{\H,T}-\nabla^{\H,T,*})u_1,u_2)_{L^2,T}=0.\ee
		By (\ref{last644}) and (\ref{last645}),  for any $u_1\in\tir_{k,T}^*\H(\bz_1,\F_1)(T),u_2\in\te_{k,T}\H(\bz_2,\F_2)(T)$, 
		\be\label{last646}|(K_Tu_1,u_2)_{L^2,T}|\leq \frac{C}{T^2}\|u_1\|_{L^2,T}\|u_2\|_{L^2,T}.\ee
		
		The lemma then follows from (\ref{last642}), (\ref{last6422}) and (\ref{last646}).

	\end{proof}
	
	Let $D^{\delta,T}_t=\nabla^{\delta,T}-\nabla^{\delta,T,*}+\sqrt{t}(d^{\bz}-\delta_T^{\bz,*})$, $D^{\H,T}_t=\nabla^{\H,T}-\nabla^{\H,T,*}+\sqrt{t}(d^{\bz}-\delta_T^{\bz,*})$. It follows from Corollary \ref{last40} and Lemma \ref{last60} that
	\begin{cor}\label{last70}
		For $t\geq 1$,
		\[|\Tr_s\left(N f'(D^{\delta,T}_t)-Nf'(D^{\H,T}_t)\right)|\leq \frac{Ce^{T}}{\sqrt{t}T^2}\]
		for some $(T,\s)$-independent $C>0$.
		
		For $t>0$,
		\[|\Tr_s\left(Nf'(D^{\delta,T}_t)-Nf'(D^{\H,T}_t)\right)|\leq C'e^{-2T}t\]
		for some $(T,\s)$-independent $C'>0.$
	\end{cor}
	\begin{proof}

		Let $K^\delta_{T}=\nabla^{\delta,T}-\nabla^{\delta,T,*}, K^{\H}_{T}=\nabla^{\H,T}-\nabla^{\H,T,*}$ , $K_{T}=K^{\delta}_{T}-K^{\H}_{T}.$

		By Lemma \ref{last60}, \be\label{last701}\|K_{T}\|\leq \frac{C}{T^2}.\ee
		
		\def\bu{\bullet}
		Let $\gamma$ be the oriented contour given by $\{z\in\C: |\Re(z)|=1\}.$
		Then by  Corollary \ref{last40},  when $T$ is large (c.f. \cite[Theorem 2.13]{bismut1995flat}), for ``$\bullet$"$=$``$\H$" or ``$\delta$"
		\[f'(D^{\bu,T}_t)=\int_{\gamma}f'(\l)(\l-D^{\bu,T}_t)^{-1}d\l.\]
		For $\l\in\gamma$, let $V=d^{\bz}-\delta_T^{\bz,*}$, then
		\be\label{last71}
		\left(\lambda-D^{\bu,T}_t\right)^{-1}=\left(1-\left(\lambda-\sqrt{t}V\right)^{-1} K^{\bu}_{T}\right)^{-1}\left(\lambda-\sqrt{t}V\right)^{-1} .
		\ee
		By Corollary \ref{last40} and Lemma \ref{tri}, for $\l\in\gamma$
		\be\label{last72}
		\|\left(\lambda-\sqrt{t}V\right)^{-1}-\frac{\P^{\ker V}}{\l}\|\leq \frac{Ce^{T}|\l|}{\sqrt{t}};
		\ee
		or
		\be\label{last721}
		\|\left(\lambda-\sqrt{t}V\right)^{-1}-\frac{\P^{\ker V}}{\l}\|\leq2.
		\ee
		Also
		\begin{align}\begin{split}\label{last73}
				&\ \ \ \ \left(1-\left(\lambda-\sqrt{t}V\right)^{-1} K^{\bu}_{T}\right)^{-1}=  \sum_{i=0}^{\dim S}\left(\left(\lambda-\sqrt{t}V\right)^{-1} K^\bu_{T}\right)^i,
			\end{split}
		\end{align}
		Proceeding as in the proof of \cite[Theorem 2.13]{bismut1995flat}, by (\ref{last701}), (\ref{last71}), (\ref{last72}), (\ref{last721}) and (\ref{last73})
		\begin{align}\begin{split}
				&\ \ \ \ |\Tr_s\Big(N\big(f'(D^{\delta,T}_t)-f'(\P^{\ker V}K^{\delta}_{T}\P^{\ker V})-f'(D^{\H,T}_t)+f'(\P^{\ker V}K^{\H}_{T}\P^{\ker V})\big)\Big)|\\
				&\leq\frac{C e^{T}}{\sqrt{t}T^2}.
			\end{split}
		\end{align}
		While by \cite[Proposition 1.3]{bismut1995flat}, 
		\[\Tr_s(Nf'(\P^{\ker V}K^{\delta}_{T}\P^{\ker V}))=\Tr_s(Nf'(\P^{\ker V}K^{\H}_{T}\P^{\ker V}))=\Tr_s(Nf'(0)).\]
		Hence, when $t\in[1,\infty)$,
		\[|\Tr_s(Nf'(D^{\delta,T}_t))-\Tr_s(Nf'(D^{\H,T}_t))|\leq \frac{Ce^{T}}{\sqrt{t}T^2}.\]
		Although $\nabla^{\delta,T}$ is not flat, the argument in \cite[Proposition 1.3]{bismut1995flat} still works. Hence, we have
		\be\label{eq101}\Tr_s\left(Nf'(D^{\delta,T}_0)-Nf'(D^{\H,T}_0)\right)=0.\ee
		
		Moreover, by a straightforward computation,
		\begin{align}\begin{split}\label{eq99}
				\frac{\p}{\p t}\Tr_s\left(Nf'(D^{\delta,T}_t)-Nf'(D^{\H,T}_t)\right)=\Tr_s\left((d^{\bz}\delta_T^{\bz,*}-\delta_T^{\bz,*}d^{\bz})\big(\tf'((D^{\delta,T}_t)^2)-\tf'((D^{\H,T}_t)^2)\big)\right).
			\end{split}
		\end{align}
		Recall that $\tf(a)=(1+2a)e^a$, hence $\tf(a^2)=f'(a).$
		
		By Corollary \ref{smalleigen},
		\begin{align}\begin{split}\label{eq98}
				\|(d^{\bz}\delta_T^{\bz,*}-\delta_T^{\bz,*}d^{\bz})(\tf'((D^{\delta,T}_t)^2)-\tf'((D^{\H,T}_t)^2))\|\leq Ce^{-2T}.
			\end{split}
		\end{align}
		
		By (\ref{eq101}), (\ref{eq99}) and (\ref{eq98}),
		\[|\Tr_s\left(Nf'(D^{\delta,T}_t)-Nf'(D^{\H,T}_t)\right)|\leq C'e^{-2T}t.\]
		
	\end{proof}
	Let
	$$\begin{aligned} &\ \ \ \ \T_f\left(d^{\bz}+\nabla^{\delta,T}\right)\\
		&=-\int_0^{\infty}\left(\varphi \operatorname{Tr}_s\left(\frac{1}{2} N f^{\prime}\left(D^{\delta,T}_t\right)\right)-\frac{1}{2} \chi'(Z,F)-\frac{d(\Omega_{\sm}(\bm,\F)(T))-\chi'(Z,F)}{2} f^{\prime}\left(\frac{i \sqrt{t}}{2}\right)\right) \frac{d t}{t},\end{aligned}$$
	and 
	$$\begin{aligned} &\ \ \ \ \T_f\left(d^{\bz}+\nabla^{\H,T}\right)\\
		&=-\int_0^{\infty}\left(\varphi \operatorname{Tr}_s\left(\frac{1}{2} N f^{\prime}\left(D^{\H,T}_t\right)\right)-\frac{1}{2} \chi'(Z,F)-\frac{d(\Omega_{\sm}(\bm,\F)(T))-\chi'(Z,F)}{2} f^{\prime}\left(\frac{i \sqrt{t}}{2}\right)\right) \frac{d t}{t}.\end{aligned}$$
	
	\begin{cor}\label{last80}
		$\lim_{T\to\infty}\left(\T_f\left(d^{\bz}+\nabla^{\delta,T}\right)-\T_f\left(d^{\bz}+\nabla^{\H,T}\right)\right)=0.$
		That is,
		\[\lim_{T\to\infty}\left(\T_{\sm}(T^H\bm,g^{T\bz},h^{\F})(T)-\T_f\left(d^{\bz}+\nabla^{\H,T}\right)\right)=0.\]
	\end{cor}
	\begin{proof}
		By the second inequality in Corollary \ref{last70},
		\be
		\int_{0}^{e^{2T}/T^2}| \operatorname{Tr}_s\left(\frac{1}{2} N f^{\prime}\left(D^{\H,T}_t\right)\right)- \operatorname{Tr}_s\left(\frac{1}{2} N f^{\prime}\left(D^{\delta,T}_t\right)\right)|\frac{dt}{t}\leq \frac{C}{T^2}.
		\ee
		By the first inequality in Corollary \ref{last70},
		\be
		\int_{e^{2T}/T^2}^\infty| \operatorname{Tr}_s\left(\frac{1}{2} N f^{\prime}\left(D^{\H,T}_t\right)\right)- \operatorname{Tr}_s\left(\frac{1}{2} N f^{\prime}\left(D^{\delta,T}_t\right)\right)|\frac{dt}{t}\leq \frac{C}{T}.
		\ee

	\end{proof}

	\begin{proof}[Proof of Theorem \ref{int6p}]Since $\nabla^{\H_2,T}\tp_{k,T}=\tp_{k,T}\nabla^{\H_1,T}$ and $\tp_{k,T}=\te_{k,T}^{-1}d^{\bz}\tir^{-1}_{k,T}$, one has $[\nabla^{\H,T},d^{\bz}]=0.$
		
		It follows from Proposition \ref{last4p1}, \cite[Lemma 2.2]{zhu2015gluing} and \cite[Theorem 7.37]{goette2001morse} that in $Q^{\mS}/Q^{\mS}_0$,
		$\lim_{T\to\infty}\T_f\left(d^{\bz}+\nabla^{\H,T}\right)-\T(T)=0.$
		The theorem then follows from Corollary \ref{last80}.

	\end{proof}
	
	\bibliography{lib}
	
	\bibliographystyle{plain}
	
\end{document}